\documentclass[12pt, reqno]{amsart}
\usepackage{amsmath}
\usepackage{amsfonts}
\usepackage{amssymb}
\usepackage[english]{babel}
\usepackage[dvipsnames]{xcolor}
\allowdisplaybreaks[2]
\usepackage{comment}
\usepackage{calrsfs}
\usepackage{graphicx}

 \usepackage{hyperref}
 \usepackage{enumitem}

\newcommand{\step}[1]{\par\medskip\noindent\bf#1\rm}
\theoremstyle{plain}
\newtheorem {theorem}{Theorem}[section]
\newtheorem {lemma}[theorem]{Lemma}

\newtheorem {corollary} [theorem]{Corollary}

\newtheorem {proposition} [theorem]{Proposition}
\theoremstyle{definition}

\newtheorem{remark}[theorem]{Remark}

\theoremstyle{remark}

\numberwithin{equation}{section}

\newcommand{\R}{\mathbb R}

\renewcommand{\H}{H}

\newcommand{\ol}{\overline}
\renewcommand{\t}{\tau}

\renewcommand{\epsilon}{\varepsilon}
\renewcommand{\phi}{\varphi}

\newcommand{\p}{\partial}

\newcommand{\e}{\varepsilon}
\renewcommand{\theta}{\vartheta}

\newcommand{\C}{\mathbb C}
\renewcommand{\a}{\alpha}
\newcommand{\z}{\zeta}
\newcommand{\wh}{\widehat}

\newcommand{\s}{\sigma}

\newcommand{\barint}
{\rule[.036in]{.12in}{.009in}\kern-.16in \displaystyle\int}

\usepackage{aurical}

\usepackage[dayofweek]{datetime}

\long\def\MSC#1\EndMSC{\def\arg{#1}\ifx\arg\empty\relax\else
     {\par\narrower\noindent
     {\small\it 2010 Mathematics Subject Classification.} \small #1\par}\fi}

\long\def\KEY#1\EndKEY{\def\arg{#1}\ifx\arg\empty\relax\else
     {\par\narrower\noindent
     {\small\it Keywords and Phrases.} \small #1\par}\fi}


\begin{document}

\date{\today} 
\title[Trace theorem     (\today )]{ A  Trace theorem  for  Martinet--type vector
fields  }

\author[D.~Gerosa]{Daniele Gerosa}
\address{Center for Mathematical Sciencies, Lund University  (Sweden)}
\email{daniele.gerosa@math.lu.se}

\author[R.~Monti]{Roberto Monti}
\address{Dipartimento di Matematica ``Tullio Levi-Civita''. Universit\`a di
Padova (Italy)}
\email{monti@math.unipd.it}

\author[D.~Morbidelli]{Daniele Morbidelli}
\address{Dipartimento di Matematica. Universit\`a di Bologna  (Italy)}
\email{daniele.morbidelli@unibo.it}

\begin{abstract} In $\R^3$ we consider the vector fields 
\[
 X_1 =\frac{ \partial }{\partial x},\qquad 
  X_2 =\frac{ \partial }{\partial y}+ |x|^\alpha \frac{ \partial }{\partial z},
\]
 where $\alpha\in\left[1,+\infty\right[$.   Let $\R^3_+
=\{(x,y,z)\in\R^3: z\geq 0\}$  be the (closed)
upper half-space and let $f\in C^1
(  \R ^3_+  )$
be a function such that $X_1f, X_2f \in L^ p(\R^3_+)$ for some $p>1$.
In this paper, we prove that the restriction of $f$ to the  plane $z=0$ belongs
to a suitable Besov space that is defined using the Carnot-Carath\'eodory metric
associated with $X_1$ and $X_2$ and the related perimeter measure. 
\end{abstract}

\maketitle

  

 \renewcommand{\C}{\mathbb C}

\renewcommand{\H}{H^n}

\newcommand{\B}{\mathcal B}

\renewcommand{\Re}{\mathrm{Re}}

\newcommand{\nablalfa} {X} 

\section{Introduction}

By a classical  result due to  Gagliardo \cite{Gagliardo57}, for any  $p>1$ 
and any bounded open set 
 $\Omega\subset\R^n$ with smooth boundary   there is a constant $C>0$
such that for any  
function $f\in C^1(\bar \Omega)$  the following \emph{trace estimate} holds:
\begin{equation}\label{gaglia}
 \int_{\p\Omega\times\p\Omega}
\frac{|f(x)-f(y)|^p}{|x-y|^{n+ps}}
d\mathcal{H}^{n-1}(x)d\mathcal{H}^{n-1}(y)  
\leq C\int_\Omega|\nabla f(x)|^p dx,
\end{equation}
where $s=1- 1/p$. The inequality extends to Sobolev functions,  showing 
that traces of $  W^{1,p} $-functions 
are well defined and have a fractional  
order of differentiability  $1-1/p$ at the boundary $\p\Omega $.

In this paper, we prove a similar trace estimate in a setting where
the gradient  of~$f$ in the right-hand side of 
\eqref{gaglia} 
is replaced by a subelliptic  gradient that, at some point of the boundary, may
be 
``tangential''.
In $\R^3$ we consider the vector fields
\[
 X_1 =\frac{ \partial }{\partial x},\qquad 
  X_2 =\frac{ \partial }{\partial y}+ |x|^\alpha \frac{ \partial }{\partial z},
\]
where    $\alpha \in\left[1,+\infty\right[$ is a real parameter. When $\alpha =
2$ the distribution of planes spanned by $X_1$ and $X_2$ is known as
Martinet-distribution. We denote the  $X$-gradient of a function $f\in C^1(\R^3)$ by 
$\nablalfa f : = (X_1 f, X_2f)$. 

Let $\R^3_+ =\{(x,y,z)\in\R^3: z\geq 0\}$ be the closed upper-halfspace
and $\Sigma = \R^2  =\{(x,y,z)\in\R^3: z=0\}$ its boundary.  
The plane $\Sigma$ is characteristic at all points where $x=0$, in the sense
that both
the vector fields
$X_1$ and $X_2$ are tangent to $\Sigma$, here.

According to a general procedure
introduced in \cite{GarofaloNhieu96} and studied in \cite{MontiSerraCassano}, the vector fields $X_1$ and
$X_2$ induce on $\Sigma$ a natural surface  measure, known as $X$-perimeter
measure.
 In the present setting, this $X$-perimeter measure is
\begin{equation} \label{millo}
  \mu = |x|^\alpha \mathcal L^2 ,
\end{equation}
where $\mathcal L^2$ is the Lebesgue measure in the plane.

We denote by $d$ the Carnot-Carath\'eodory metric on $\R^3$ induced by
$X_1,X_2$ and by $B(q,r)$ the metric ball centered at $q\in\R^3$ with radius
$r>0$. With abuse of notation, we   identify $u\in\R^2$ with $(u,0)\in\R^3$.
 

\begin{theorem} 
\label{tello} 
Let $  \alpha \in\left[1,+\infty\right[$, 
$  p\in\left]1,+\infty\right[  $ and   $s=1-1/p$. There exists a constant
$C>0$ depending on $\a$ and $p$   
such that any function $f\in  C ^1  (  \mathbb{R}^3_+   )  $
satisfies 
\begin{equation} \label{equa}
\int_{\mathbb{R}^2
\times \mathbb{R}^2} \frac{|f(u,0) - f(v,0)|^p  }{d(u,v)^{ps} \mu(B(u,d(u,v)))}
\, d \mu(u) \, d \mu(v)  \le C 
\int_{\mathbb{R}^3_+ } 
|\nablalfa f (x,y,z)|^p dx dy dz.
\end{equation}
\end{theorem}

The  Besov seminorm in the left-hand side is defined in terms of
the metric $d$ and of the measure $\mu$. When $d$ is the
standard metric and $\mu$ is the Lebesgue measure,
the seminorm reduces to the   one in the left-hand side of \eqref{gaglia}.

This seminorm was first  introduced by Danielli,
Garofalo and Nhieu in
\cite{DanielliGarofaloNhieu06}, where   a metric 
approach to the problem is developed.  
The authors prove trace and lifting theorems for   $(\epsilon,\delta)$-domains 
with \emph{Ahlfors regular} boundary.
The  $(\epsilon,\delta)$-property is in general difficult to check because of
the presence of boundary characteristic points. For systems of H\"ormander
vector
fields with step 3, it may fail even for ``flat'' or analytic boundaries, see
\cite{MontiMorbidelli05}.
In a companion paper \cite{MontiMorbidelli18}, we are able to show the
$(\epsilon,\delta)$-property for a different family of vector fields related to
generalized Siegel domains.
The Ahlfors
regularity of the measure $\mu$ in \eqref{millo} will be studied in Section~\ref{AL}.

The classical proof of \eqref{gaglia} by Gagliardo  relies on an elegant
construction of families
of curves transversal to the surface  
$\partial\Omega$ and connecting pairs of points on the boundary. The
estimate is achieved by an integration of the gradient of the function along
such curves. This technique can be extended to  the subelliptic setting if
$\p\Omega$ does not contain {\emph {characteristic points}}.
Indeed,  in the noncharacteristic case the construction of transversal
\emph{horizontal curves} is  easy  because at any noncharacteristic
boundary point there is at least one vector field transversal to the tangent
space to the boundary. 
Trace inequalities in this setting are proved   by Berhanu and Pesenson
in \cite{BerhanuPesenson}, by Bahoury, Chemin and Xu in \cite{BahouriChaminXu05}
for
vector fields of step 2 and by the authors for general 
H\"ormander vector fields in \cite{MontiMorbidelli02}.

In the characteristic case, the construction of horizontal curves entering the
domain from
boundary points is much more delicate. Some  trace
theorems are known also in this case,  mainly in two
classes of examples. The first one is  the
Heisenberg group, see the contribution by Bahouri, Chemin and Xu
\cite{BahouriCheminXu09} for some  characteristic surfaces.
A second class of examples  is that of \emph{diagonal} vector fields, i.e., a
system of $n$  vector fields in~$\R^n$ of the form $X_i = w_i(x)
\frac{\partial}{\partial x_i}$, $i=1,\ldots,n$, with suitable weights~$w_i$. See
the results of Franchi \cite{Franchi86} and  the authors
\cite{MontiMorbidelli02}.

In this paper, we are able to  deal with the following three
difficulties:
\begin{itemize} 
\item[--]  the  plane $z=0$  contains characteristic points and actully a
whole line, the $y$-axis; 
\item[--]  the vector fields can have arbitrarily large step, depending on
$\alpha\geq 1$;
\item[--]  the vector fields are not of diagonal type.
 \end{itemize}
In a future work, we plan to generalize our results to more general surfaces and
to more general  families of vector fields.

In our proof of \eqref{equa}, it is enough to estimate  the difference
$f(u,0)-f(v,0)$  for noncharacteristic points. However,   both $(u,0)$ and
$(v,0)$
  may be arbitrarily close to 
 the characteristic line.  
The choice of the curves connecting them is rather 
 delicate and must take into account ``how much'' close to the
characteristic set the points are.
Once the correct construction is devised, the trace estimate is obtained by
integrating the subelliptic gradient along such curves and using the Minkowski
and Hardy integral inequalities.  
The correct estimate of the Besov seminorm 
must be  split in several sub-cases and each of them requires
a separate effort.

The argument requires a precise description of the size of the
Carnot-Carath\'eodory  balls of the distance $d$ associated with  the
vector fields $X_1, X_2$. Since $\a$ can be noninteger, we cannot use the
ball-box theorems of Nagel, Stein and Wainger
\cite{NagelSteinWainger}. 
For this reason, in Section \ref{ullo} we   give a self contained proof of
the ball-box estimate for  $d$, which  has an independent
interest.

 
\step{Notation}. By $C_\alpha>0$ we denote a constant depending on $\alpha\geq
1$
that may change from line to line.
By $C_{\alpha,p}>0$ we denote a constant depending $\alpha\geq 1 $ and $p>1$
that may
change from line to line.
For $a,b>0$, we use the standard notation     $a\simeq b$ 
meaning that 
 $a\leq C  b$ and $b\leq C  a$ for an absolute constant $C$ that may depend on
$\alpha$ and/or $p$.

 \section{Structure of the metric}
\label{ullo}

Let $d$ be the Carnot-Carath\'eodory  distance associated with the vector
fields $X_1 =\p_x$  and $X_2=\p_y+|x|^\a\p_z$. 
The construction of $d$ is well-known and can by found in~\cite{NagelSteinWainger}.

When  $\a=2$,  the vector fields 
$X_1$ and $X_2$ span a distribution of 2-planes in $\R^3$ known as 
\emph{Martinet}-distribution.
When  $\a $ is an even number,  the vector fields satisfy the H\"ormander
condition with
step $\a+1$ and the structure of metric balls follows from 
\cite{NagelSteinWainger}. 

When $\alpha$ is not even, the results of \cite{NagelSteinWainger}
cannot be used.
For this reason, we give here a self-contained proof of the relevant estimates.
The case $\a=1$ of the familiar Heisenberg group is
not included in our discussion. However, with some  minor adaptations,   the
results of this section  hold
verbatim   for vector
fields of the form $X_1 = \p_x $ and $X_2=\p_y+|x|^{\a-1}x \p_z$, including the
Heisenberg vector fields in
the limit case $\a=1$.

By the particular structure of the vector fields, the distance $d$ possesses the
following invariance properties
\begin{subequations}\begin{align}
\label{cip}  & 
 d((x,y,z), (x', y', z',))=d((x,y+\eta ,z+\zeta ),(x',y'+\eta ,z'+\zeta )),
 \\ & d((rx,xy,r^{\a+1}z), ( rx',ry',r^{\a+1} z' ))= rd((x,y,z), (x', y', z')),
 \\&   
 \label{cetra} d((-x,y,z),(-x',y',z'))=   d((x,y,z),(x',y',z'))\color{black}
\end{align}
\end{subequations}
for all $(x,y,z), (x',y',z')\in\R^3$, $\eta,\zeta\in\R$ and $r\geq 0$.
In this section, we describe the structure of  $d$ in terms
of an equivalent  function defined  by algebraic functions.

For  $\alpha\geq 1$, we define the function 
$\delta:\R^3\times\R^3\to[0,\infty)$
\begin{equation*} 
 \delta((x,y,z),(x',y',z')) :=   |x'-x|+|y'-y|+  \min \bigg\{ 
|\zeta |^{1/(\alpha+1)},
 \frac{ |\zeta |^{1/2}}{ |x
|^{(\alpha-1)/2}}\bigg\}, 
\end{equation*} 
where we let 
   $\zeta = z - z' +|x|^{\alpha}(y' - y)$. 
In the definition above,  we agree that the minimum is $  |\zeta |^{1/(\alpha
+1)}$  if $x=0$, and  is  $ 0$   if $\zeta =0$.

\begin{theorem}\label{ddss}   For   $\alpha\geq 1$, let $d$ be the
Carnot-Carath\'eodory metric
induced on $\R^3$ by the vector-fields $X_1,X_2$.
There exists a constant $C_0>0$, depending on $\a$, such that for all $p,q \in \R^3$  we have 
 \begin{equation}
C_0^{-1} \delta(p,q) \leq  d(p,q) \leq C_0  \delta(p,q)  .
 \end{equation}  
\end{theorem}

\begin{proof} 
For $\a\geq 1$,  we will   use the equivalence
\begin{equation}
 \label{dasotto}  |u^\a-v^\a|  \geq   C_\a   
(|u|+|v|)^{\a-1}|u-v|,\quad \text{for all } u,v\in\left[0,+\infty\right[.
\end{equation} 
By the translation invariance \eqref{cip}, we can assume that  
$ p = (x,y,z)=(x,0,0)$.  We also let $q = (x',y',z')$.

\medskip

\textit{Step 1.} We first show the estimate  $\delta \leq C_0 d$. Let
$\gamma:[0,T]\to \R^3$ 
 be a horizontal curve with   $\gamma(0) = p $, $\gamma(T) = q$ and $\dot \gamma
=h_1(t) X_1(\gamma)+h_2(t)X_2(\gamma)$ with $|(h_1 , h_2)|\leq 1$ a.e.
The functions 
 \begin{equation*}
  x(t)=x+\int_0^t h_1 (s)ds=:x+\widehat x(t),\qquad y(t)=\int_0^t h_2(s)ds,
 \end{equation*}
satisfy the  estimates $|\widehat x(t)|\leq t$ and $|y(t)|\leq t$, and thus 
 \begin{equation}\label{daunno} 
  |x' - x| =|x(T)-x(0)| \leq T\qquad \text{ and }\quad |y'|= | y(T) -
y(0)|\leq T. 
 \end{equation} 

The quantity  $\zeta = z - z' +|x|^{\alpha}(y' - y)= |x|^{\alpha} y'   - z'
$ \color{black} 
satisfies  
  \begin{equation*}
\begin{aligned}
 |\zeta  | & 
  =\Big| \int_0^T (|x(s)|^\alpha  -|x|^\alpha)\dot y(s) ds  \Big|
 \leq C_\alpha  
    \int_0^T (|x|^{\alpha -1}+s^{\alpha -1}) s ds
    \leq C_\alpha ( |x|^{\alpha -1}T^2+T^{\alpha+1}),
\end{aligned}
 \end{equation*}
 and this implies 
that either
  $|\zeta | \leq C_\alpha |x|^{\alpha -1}T^2    $ or
$| \zeta| \leq C_\alpha T^{\alpha +1} $. This is equivalent to  
\begin{equation*}  \min\Big\{ \frac{|\zeta  |^{1/2}}{|x|^{(\alpha -1)/2}} , 
|\zeta |^{1/(\alpha +1)}\Big\} \leq C_\alpha T, 
\end{equation*}
and this estimate  together with \eqref{daunno} concludes the proof of
\textit{Step 1}.

 \medskip

\textit{Step 2.} We prove the estimate $d\leq C_ 0  \delta$ in the case when
points are one above
the other. Namely, we claim that  $d((x ,y,z),  (x,y,z'))\leq C_ 0   \delta
((x ,y,z), 
(x ,y,z'))$ for all $x,y,z,z'\in\R$.
As above, we can assume that $y=z=0$.

We prove the claim  for $x\geq 0$  and $z>0$. The cases  $x<0$  and $z<0$ 
are analogous. 
For  $u>0$, let $\kappa:[0, 4u]\to \R^2 $ be 
 the plane curve with unit speed  which  connects the points  $(x,0)$, $(x+u,
0)$, $(x+u,u) $,  $(x,u)$   and $(x,0)$, and let  $\kappa(t)=(x(t),
y(t))$. This path encloses a square which we denote by $R_u$. Let  $t\mapsto
\gamma(t) =  (x(t), y(t), z(t))$ be the horizontal lift of $\kappa$ starting
from $z(0)=0$. By Stokes' theorem  
 \begin{equation} \label{ulla}
  \begin{aligned}
z(4u) &=\int_\kappa \xi^\alpha d\eta  
=\int_{R_u}\alpha \xi^{\alpha -1}d\xi
d\eta      
=u((x+u)^\alpha-x^{\alpha})
\geq C_\a u^2(x^{\alpha -1}+u^{\alpha -1}),
\end{aligned}
 \end{equation}
where we used \eqref{dasotto}.  By the definition of $d$, we have 
\[
  d ((x,0,0),  (x,0,z' )) \leq 4   
 \min\{ u> 0 : z(4u)=z'\}, 
\]
and, by \eqref{ulla}, the number  $u$ realizing the minimum
satisfies 
\[
 u \leq  C_\alpha \min\Big\{
  \frac{ | z '| ^{1/2}}{ |x|^{(\alpha -1)/2}}, 
  | z'| ^{1/(\alpha +1)}\Big\}.
\]
This proves the claim.
%
%

\medskip

\textit{Step 3.} We prove the estimate $d\leq C_0   \delta$
for arbitrary points $p = (x,0,0)$ and $q = (x',y',z')$. By the  triangle
inequality we have
\begin{equation}\label{secchio} 
 \begin{split}
 d(p, q) & \leq d\big(p, \mathrm{e}^ {(x'-x)X_1+y' 
X_2}(p)\big) +
 d\big( \mathrm{e}^ {(x'-x)X_1+y' 
X_2}(p), q\big)
\\
&  \leq  
 |x-x'| +|y'|+ d\big( \mathrm{e}^ {(x'-x)X_1+y' 
X_2}(p),q\big)
\\
&  \leq  
 \delta(p,q)  + d\big( \mathrm{e}^ {(x'-x)X_1+y' 
X_2}(p),q\big),
\end{split}
\end{equation}
where we adopt the standard notation $e^{Z}(p)$ or $\exp(Z)(p)$
to denote the value at time $1$ of the integral curve of  the vector 
field $Z$ starting from $p$ at $t=0$. An easy computation shows that 
\[
  \mathrm{e}^ {(x'-x)X_1+y' 
X_2}(p) =\Big( x',y',  y' \int_0^1 |x+s(x'-x)|^\alpha ds\Big),
\]\color{black}
i.e., the point is above $q$. By the  \textit{Step 2}, we have
\[
 d\big( \mathrm{e}^ {(x'-x)X_1+y' 
X_2}(p),q\big) \leq C_ 0 \delta \big( \mathrm{e}^ {(x'-x)X_1+y' 
X_2}(p),q\big).
\] 
Now, letting 
\[
  \zeta = z'   -|x|^{\alpha } y' \qquad \textrm{and}\qquad 
  \zeta ' = z' -y'  \int_0^1 |x+s(x'-x)|^\alpha ds,
\]
to conclude the estimate  it suffices to show that  
\begin{equation}\label{dappino} 
\begin{aligned}
  \min  \bigg\{  |\zeta' |^{1/(\alpha+1)} ,
 \frac{|\zeta ' |^{1/2 }}{ |x' |^{(\alpha -1)/ 2} } 
 \bigg\}
   \leq C_{\alpha} \bigg( |x-x'|+|y'|+ 
   \min \bigg\{ \ |\zeta  |^{1/(\alpha
   +1)},
 \frac{ |\zeta  |^{1/2}}{ |x |^{(\alpha-1)/2
}}\bigg\}\bigg).
   \end{aligned}
\end{equation} 
First of all we have  
\begin{equation}\label{pejo} 
\begin{split}
| \zeta'|  & =   \Big| z'  
 -  y '  \int_0^1 |x+s(x' -x)|^\alpha ds \Big|
\\
&
 \leq |\zeta |+|y' | \, \Big|\int_0^1\Big( |x+s(x' - x)|^\alpha -
|x|^\alpha)\Big) ds
\Big| \leq C_\alpha(|\zeta|+\omega), 
\end{split}
\end{equation}
where we let $\omega =  |y'| (|x|+|x' |)^{\alpha -1}|x' - x|$. 
To prove \eqref{dappino} we distinguish the following two cases:

\noindent\textit{Case A}: $|\zeta |^{1/(\alpha +1)}\leq
\dfrac{|\zeta |^{1/2}}{|x|^{(\alpha - 1)/2}}$, or, equivalently, $|x|\leq
|\zeta|^{1/(\alpha+1)}$.

\noindent\textit{Case B}:  $|\zeta |^{1/(\alpha +1)}\geq
\dfrac{|\zeta |^{1/2}}{|x|^{(\alpha - 1)/2}}$, or, equivalently, $|x|\geq
|\zeta|^{1/(\alpha+1)}$.

\medskip

In the \textit{Case A}, the claim \eqref{dappino} is implied by  
\begin{equation*}
 \min\Big\{(|\zeta|+\omega)^{1/(\alpha+1)},
\frac{(|\zeta|+\omega)^{1/2}}{|x'|^{(\alpha
-1)/2}} \Big\}
 \leq C_{\alpha} \Big(  |x-x'|+ |y' |+|\zeta|^{1/(\alpha+1)}\Big) .
\end{equation*}
The estimate of $|\zeta|^{1/(\alpha +1)}$ is trivial. The quantity
$\omega$ is estimated in the following way: 
\[
\begin{split}
  \omega & \leq C_\alpha  |y'| (|x|+|x-x'|)^{\alpha -1}|x - x'|
 \leq C_\alpha \big(  
        |y'|^{\alpha +1}+ | x|^{\alpha +1}+|x-x'| ^{\alpha +1} \big) 
\\
 &\leq 
C_\alpha \big( 
   |y' |^{\alpha +1}+ 
|\zeta |+|x-x'| ^{\alpha +1}\big),
\end{split}
\]
and the claim follows.

 \medskip

In the \textit{Case B}, the claim  \eqref{dappino} is implied by  
\begin{equation*}
 \min\Big\{(|\zeta|+\omega)^{1/(\alpha+1)},
\frac{(|\zeta|+\omega)^{1/2}}{|x'|^{(\alpha
-1)/2}} \Big\}
 \leq C_{\alpha} \Big(  |x-x'|+ |y' |+
 |\zeta  |^{1/2}  / |x |^{(\alpha-1)/2 } \Big)  . 
\end{equation*}
 
\medskip 
\noindent\textit{Sub-case B1}: $|x'-x|\leq\frac 12 |x|$.
In this sub-case, we have 
$ |x'|\simeq |x|$ and thus  the  term with $\zeta$ is easily
estimated, because
$
 \frac{|\zeta|^{1/2}}{|x'|^{(\alpha-1)/2}}
\simeq
\frac{|\zeta |^{1/2}}{|x|^{(\alpha-1)/2}}.
$

We estimate  the term with $\omega$. From $\omega \simeq |x|^{\alpha
-1}|y'|\,|x-x'|$ we deduce that  
\begin{equation*}
\begin{aligned}
   \frac{\omega^{1/2}}{|x'|^{(\alpha-1)/2}}\simeq \frac{\big(|x|^{\alpha
-1}|y'|\,|x-x'| \big)^{1/2}}{|x|^{(\alpha-1)/2}}\simeq
|y'|^{1/2}|x-x'|^{1/2},
\end{aligned}
\end{equation*}
which is smaller than $|x-x'|+|y'|$, as required.

\medskip 
\noindent\textit{Sub-case B2}: $|x-x'|>\frac 12 |x|$. We claim that 
\begin{equation*}
 |\zeta |^{1/(\alpha +1)}\leq C_\alpha \Big( 
|x-x'|+\frac{|\zeta|^{1/2}}{|x|^{(\alpha-1)/2}} \Big) .
\end{equation*}
Indeed,   the function $h(s) = s+\frac{|\zeta|^{1/2}}{ s^{(\alpha-1)/2}}$ attains
the minimum   on $(0,\infty)$ at the point  
$s_{\textrm{min}}\simeq |\zeta |^{1/(\alpha+1)}$.

To end the discussion of the \textit{Sub-case B2}, we estimate the  term
with $\omega$:
\begin{equation*}
\begin{aligned}
\omega  \leq 
|y' |(2|x|+|x-x'|)^{\alpha -1}|x-x'|\leq  C   |y' |\,|x-x'|^\alpha,
\end{aligned}
\end{equation*}
and  the  estimate $ 
 \omega^{1/(\alpha +1)} \leq C_\alpha(  |x-x'|+|y'|) $ follows.

\medskip

This concludes the proof of Theorem \ref{ddss}.
\end{proof}

\begin{remark}
 For all points $(x,y),(x',y')\in\R^2$ we have the equivalence
 \begin{equation}\label{piana} 
\begin{aligned}
   d((x,y,0), (x',y',0))&\simeq  |x-x'|+|y-y'| +|x|^{1/2}|y'-y|^{1/2}.
\end{aligned}
 \end{equation} 
To prove this, we start  from
\[
 \delta((x,y,0), (x', y', 0))=|x-x'|+|y-y'|+\min\big\{|x|^{\a/(\a+1)}|y'-y|^{1/(\a+1)},  
 |x|^{1/2}|y'-y|^{1/2 }\big\},
\]
and we observe that  the minimum is equivalent to the second term, because  
\begin{equation*}
|x|^{1/2}|y'-y|^{1/2 }\leq C_\alpha \big(     |y-y'| +
|x|^{\a/(\a+1)}|y'-y|^{1/(\a+1)} \big) .
\end{equation*}
\qed
 \end{remark}

 We rephrase the estimates in Theorem \ref{ddss} as a  ball-box  theorem.
For a fixed point $p = (x,y,z) \in\R^3$, define the mappings
$\Phi_1(p;\cdot),\Phi_2(p;\cdot):\R^3\to\R^3$:
\begin{equation*}
\begin{aligned}
 \Phi_1(p; u) &  =
  \Phi_1(  u)= \Big(x+u_1, y+u_2 , z+|x|^\alpha u_2 + |x|^{\alpha -1 }u_3\Big),
\\  
\Phi_2(p; u) &    =   \Phi_2(  u)= \Big(x+u_1, y+u_2 , z+|x|^\alpha u_2 + 
u_3\Big).
  \end{aligned}
\end{equation*}
We let 
$\| u\|_{1,1,2} = \max\{|u_1|, |u_2|, |u_3|^{1/2}\}$ and 
$
\| u\|_{1,1,\alpha +1}= \max\{|u_1|, |u_2|, |u_3|^{1/(\alpha+1)}\}$, and we
define the boxes
\[
 \begin{split}
  B_1(p, r) =\{\Phi_1(p;u):\| u\|_{1,1,2}<r\} 
\quad 
\textrm{and} 
\quad
 B_2  (p, r) =\{\Phi_2(p; u):\| u\|_{1,1,\alpha +1}<r\}.
 \end{split}
\]
Let $C_0>1$ be a
constant such that 
$C_0^{-1}{\delta}\leq  {d}\leq C_0{\delta}$ globally.

\begin{corollary}\label{boxxo} 
Let  $\eta >0$. There are constants $b_1(\eta  )$ and
$b_2(\eta )$ such that for all $p=(x,y,z) \in\R^3$ and $r>0$ we have:
 \begin{itemize}
  \item[i)]  if $|x|\geq \eta r $ then 
  \begin{equation}\label{1212} 
B_1\big(p, C_0^{-1}  r \big)\subset B (p,  r ) \subset 
B_1 \big(p, b_1(\eta ) r  \big) ; 
  \end{equation}
  
  \item[ii)] if $r\geq \eta|x|$, then 
  \begin{equation}\label{1313} 
B_2\big(p, C_0^{-1}  r  ) \subset  B (p, r )
\subset 
B_2\big(p, b_2(\eta ) r \big).
  \end{equation}
 \end{itemize}
\end{corollary}

 \begin{proof}  \textit{Step 1.}  We claim that for all $p$ and $r$ we
have:
 \begin{equation*}
  B_1 ( p, C_0^{-1}r)\cup B_2 (p , C_0^{-1}r)
  \subset B     (p,  r)
  \subset   
    B_1 (p , C_0 r)\cup B_2 ( p, C_0 r).
 \end{equation*} 
Indeed, letting $\zeta = z-z'+|x|^\alpha (y'-y)$,  we have
 \begin{align*}
  (x',y',z')\in B_1(p,r  ) \quad &\Leftrightarrow \quad
\max \Big\{ |x-x'|,|y-y'|,\Big( \frac{|\zeta  |}
 {|x|^{\alpha -1}}\Big)^{1/2}\Big\}<r,
\\ \label{mamma} 
  (x',y',z')\in B_2(p,r  ) \quad &\Leftrightarrow \quad
\max \Big\{ |x-x'|, |y-y'|,|\zeta |^{1/(\alpha+1)}\Big\} <r.
 \end{align*} 
This means that   $  (x',y',z')\in (B_1\cup B_2)(p,
r)$ if and only if $  \delta((x,y,z), (x',y',z')) < r$.
Then \textit{Step 1} is concluded thanks to Theorem \ref{ddss}. 
The argument also proves the inclusions in the left-hand side of  \eqref{1212}
and \eqref{1313}.

\medskip\noindent\textit{Step 2. } We prove the  inclusion in the right-hand
side of  \eqref{1212}.
Let $|x|>\eta r$ 
and let $(x', y', z')\in B(p,  r)$. By  \textit{Step 1} we have  
$(x',y',z')\in  B_1(p,C_0 r)\cup B_2 (p,C_0 r)$.  
To conclude the proof it suffices to show that there is a constant 
$b_1(\eta )>0$ so that 
the following implication holds: 
\begin{equation}\label{impli} 
 \left\{\begin{aligned}
         &|x|>\eta r   
         \\& \min \Big\{\frac{|\zeta|^{1/2}}{|x|^{(\alpha
-1)/2}},|\zeta|^{1/(\alpha
+1)}  \Big\}
         \leq C_0  r \quad (*)
        \end{aligned}
 \right.
 \quad \Rightarrow \quad \frac{|\zeta|^{1/2}}{|x|^{(\alpha -1)/2}}\leq
b_1(\eta )   r.
\end{equation}

If 
 $\frac{|\zeta |^{1/2}}{|x|^{(\alpha -1)/2}}\leq |\zeta |^{1/(\alpha +1)} $,
there is
nothing to prove and we can   choose $b_1(\eta )=C_0 $. 
In  the case  
$\frac{|\zeta|^{1/2}}{|x|^{(\alpha -1)/2}}\geq |\zeta |^{1/(\alpha +1)} $,  
inequality $(*)$ reads $
 |\zeta|^{1/(\alpha +1)}\leq C_0  r$ 
and we have:
 \begin{equation*}
\begin{aligned}
 \frac{|\zeta |^{1/2}}{|x|^{(\alpha -1)/2}} &\leq  
   \frac{|\zeta |^{1/2}}{\eta^{(\alpha -1)/2}  r^{(\alpha -1)/2}} 
\leq   \frac{ C_0 ^{(\alpha+1)/2}}{\eta^{(\alpha -1)/2}} \,r.
  \end{aligned}
 \end{equation*}
The proof of \textit{Step 2} is concluded, with   $b_1(\eta )
=\max\big\{C_0 , 
 \frac{ C_0^{(\alpha +1)/2}}{\eta^{(\alpha -1)/2}}\big\}$.

\medskip\noindent\textit{Step 3.} We prove the  inclusion in the right-hand
side of \eqref{1313}.
As in the \textit{Step 2}, it suffices to show the implication  
\begin{equation}\label{implicato} 
 \left\{\begin{aligned}
         &r>\eta |x|    
         \\& \min \Big\{\frac{|\zeta |^{1/2}}{|x|^{(\alpha
-1)/2}},|\zeta|^{1/(\alpha
+1)}  \Big\}
         \leq C_0  r \quad  
        \end{aligned}
 \right.
 \quad \Rightarrow \quad  |\zeta|^{1/(\alpha +1)} \leq b_2(\eta )   r.
\end{equation}
If the minimum is $|\zeta |^{1/(\alpha+1)}$, we trivially get the implication
with 
 $b_2(\eta )=C_0 $. Otherwise, we have 
 \begin{equation*}
\begin{aligned}
 C_0 r &
 \geq \min 
 \Big\{\frac{|\zeta |^{1/2}}{|x|^{(\alpha -1)/2}},|\zeta |^{1/(\alpha +1)} 
\Big\}=
 \frac{|\zeta | ^{1/2}}{|x|^{(\alpha -1)/2}}\geq  
|\zeta|^{1/2}\frac{\eta^{(\alpha -1)/2}}{r^{(\alpha -1)/2}},
\end{aligned}
 \end{equation*}
which is equivalent to $|\zeta |^{1/(\alpha +1)}\leq
\Big(\frac{C_0 }{\eta^{(\alpha -1)/2}}\Big)^{2/(\alpha +1)} r$, 
as required.
Therefore, implication \eqref{implicato} holds with
$b_2(\eta )
=\max\big\{C_0   ,
\big(\frac{C_0 }{\eta^{(\alpha -1)/2}}\big)^{2/(\alpha +1)}\big\}$.
\end{proof}

Using the previous corollary, it is immediate to get the following estimates of the Lebesgue measure of the
balls $B(p,r)$.

\begin{corollary}
\label{boxxo2} 
Let  $\eta >0$. For all $p=(x,y,z) \in\R^3$ and $r\in\left]0,+\infty\right[$ we have:
 \begin{itemize}
  \item[i)]  if $|x|\geq \eta r $ then  $\mathcal L^3  (B(p,r)) \simeq 
 r^{4}   |x|^{\alpha-1}$;
  
  \item[ii)] if $|x|\leq \eta r $ then $\mathcal L^3  (B(p,r)) \simeq  r
^{\alpha+3}$.
 \end{itemize}
The equivalence constants depend on $\alpha $ and $\eta$.
\end{corollary}

We omit the proof, which is trivially based on Corollary~\ref{boxxo}.

\section{Ahlfors' property} \label{AL}

The boundary of the half-space $\R^3_+$ is the plane $\Sigma =\R^2 = \{ (x,y,z)
\in\R^3: z=0\}$. According to the general construction of \cite{GarofaloNhieu96} and
\cite{MontiSerraCassano}, the vector-fields $X_1,X_2$ induce on $\Sigma$ a Borel
measure known as $X$-perimeter measure. We denote this measure by $\mu$.
The integral-geometric formula for this measure is the
following:
\[
 \mu(B) = \int _{B} \sqrt{\langle X_1,N\rangle^2 + \langle X_2, N\rangle^2 }
dxdy,\qquad B\subset \R^2 \textrm{ Borel set}.
\]
Above, $N = (0,0,-1)$ is the exterior normal to the boundary of  $\R^3_+$
 and $\langle
\cdot,
\cdot\rangle$ denotes the standard scalar product of $\R^3$. 
In fact, the measure $\mu$ is simply
\begin{equation} \label{mix}
 \mu  = |x|^\alpha \mathcal L^2 \quad \textrm{ on $\Sigma = \R^2$}.
\end{equation}
The metric $d$ and the balls $B(p,r)$ can be restricted to $\Sigma$.
With abuse of notation we let $\mu(B(p,r)) = \mu (B(p,r))\cap \Sigma)$.
The measure $\mu$ is Ahlfors regular in the following sense (see \cite{DanielliGarofaloNhieu06}).

\begin{proposition} \label{1:prop} 
There is a constant $C_\a>0$  such that for any $p \in \Sigma$ and for all $r>0$ we have
\begin{equation}  \label{callo}
C_\a^{-1}\frac{\mathcal{L}^3
(B(p,r))}{r} \leq  \mu (B(p , r) ) \leq C_\a \frac{\mathcal{L}^3
(B(p,r))}{r} .
\end{equation} 
\end{proposition} 

\begin{proof} Let $p= (\bar x,\bar y,0)$ and $r>0$ be such that $|\bar x|\geq
r$. The  section of the
ball $B_1(p,r)$ with the plane $\Sigma$ is 
\[
\begin{split}
 & B_1(p,r) \cap \Sigma 
\\& = \Big\{  \big(\bar x+u_1, \bar y+u_2 , 
|\bar x|^\alpha
u_2 + |\bar x|^{\alpha -1 }u_3) \in\R^3 : 
\| u\|_{1,1,2}<r, |\bar x|^\alpha
u_2 + |\bar x|^{\alpha -1 }u_3 = 0 \Big\}
\\ 
&
=\big[\bar x-r,\bar x+r\big] \times \big[\bar y - \min \{ r, r^2/|\bar x|\} ,
\bar y + \min \{ r, r^2/|\bar x|\}\big] .
\end{split} 
\]
Then, from  
\[
 \mu(B_1(p,r)) = 2\min \{ r, r^2/|\bar x|\} \int _{\bar x-r} ^{\bar x+r}
|x|^\alpha dx
\]
and from Corollary \ref{boxxo} we deduce that when $|\bar x|\geq  r$ we   have
\[
 \mu(B(p,r)) \simeq \mu(B_1(p,r)) =\frac{2r^2}{|\bar x|} \int _{\bar x-r} ^{\bar
x+r}
|x|^\alpha dx \simeq {r^3}{|\bar x|^{\alpha-1} }.
\]

On the other hand, the section of the
ball $B_2(p,r)$ with the plane $\Sigma$ is
\[
\begin{split}
 B_2(p,r) \cap \Sigma & 
= \Big\{  \big(\bar x+u_1, \bar y+u_2 , |\bar x|^\alpha
u_2 +  u_3) \in\R^3 : 
\| u\|_{1,1,\alpha+1 }<r, |\bar x|^\alpha
u_2 +  u_3 = 0 \Big\}
\\ 
&
=\big[\bar x-r,\bar x+r\big] \times \big[\bar y - \min \{ r, r^{\alpha+1} /|\bar
x|^\alpha \} ,
\bar y + \min \{ r, r^{\alpha+1} /|\bar x|^\alpha \}\big] ,
\end{split} 
\]
and thus  
\[
 \mu(B_2 (p,r)) = 2\min \{ r, r^{\alpha+1} /|\bar x|^\alpha \}\int _{\bar x-r}
^{\bar x+r}
|x|^\alpha dx.
\]
When $|\bar x|\leq  r$, from Corollary \ref{boxxo}  we deduce that 
\[
 \mu(B(p,r)) \simeq \mu(B_2(p,r)) =2r  \int _{\bar x-r} ^{\bar
x+r}
|x|^\alpha dx \simeq {r^ {\alpha+1} }.
\]
Now the claim \eqref{callo} is a consequence of Corollary \ref{boxxo2}.
  \end{proof}

 \section{Schema of the proof}

 In this section, we  outline the scheme of the proof of Theorem \ref{tello}.
We have the points  $u=(x,y,0)$, $v=(x',y',0)$ and their distance   $d =
d(u,v)$. Since $B(u,d) \subset B(v,2d)\subset B(u, 3d)$, the integral kernel 
\[
 \frac{|f(u,0) - f(v,0)|^p  }{d(u,v)^{ps} \mu(B(u,d(u,v)))}
\]
is
``almost'' symmetric in $u$ and $v$ and we can assume that 
\begin{subequations}\label{gips} 
\begin{align}  \label{yy'}
& y'\geq y\quad \text { and }
\\& 0<x'<\infty. \label{xx} 
\end{align}
\end{subequations}
Assumption \eqref{xx} can be made without loss of generality in view of  
the invariance property \eqref{cetra}.
 
We will connect the points $u$ and $v$  by a number of integral curves of the
vector fields $\pm X_1$,  $\pm X_2$, or of their sum $\pm( X_1+X_2 )$. The
correct choice depends on the following cases.

Let $ \e_0\in\left]0,1\right[$ be a  small parameter that will be fixed along
Section~\ref{pop}. We have the following cases: 

1) $d\geq \e_0|x|     $ and  $d\geq \e_0|x'| $. We call this the
\emph{characteristic  case}.

2) $d<\e_0|x|   $ and  $d<\e_0|x'| $. We call this the
\emph{noncharacteristic  case}.

3) $d< \e_0|x| $  and $d\geq \e_0|x'|    $, or viceversa.

\noindent 
In the third case, we have   $|x'| \geq |x| - |x-x'| \geq \e_0^{-1} d - d = ( \e_0^{-1}-1)
d$
because  $|x-x'| \leq d$. So this case is essentially contained in the
second one.

Theorem \ref{tello} is then reduced to the proof of \eqref{equa} when the
integration domain $\R^2\times\R^2$ is replaced by the case 1) and 2),
separately, along with the two conditions  \eqref{gips}. 
The proof for the characteristic case is  in Section  \ref{pip}.
The proof for the noncharacteristic case is in Section   \ref{pop}.

\section{Characteristic case}
\label{pip}


We connect the points $u=(x,y,0)$ and $v=(x',y',0)$ with integral
curves of the vector fields $\pm X_1$ and $\pm X_2$. Our first task is to fix
the order in the sequence of these vector fields.
We start by discussing a  first subcase of \eqref{xx}. Namely, we assume that   
\begin{equation}\label{nolog} 
  0< |x|\leq x' .
\end{equation} 
We will explain in Remark \ref{remo}  (see page  \pageref{remo}) 
how to deal with the second sub-case case $|x| >
 x' >0$. In the following, for $u,v\in\R^2$ we let $d= d(u,v)$. We define the
set
\[
 A = \big\{ (u,v) \in\R^2\times\R^2 :y\leq y', \, 0< |x| \leq x' \leq
d/\e_0\big\}.
\]

Starting from $u$, we introduce certain intermediate points interpolating $u $
and $v$. Let $\t =\tau(u,v)>0$ be the number 
\begin{equation}\label{ttaa} 
 \t = \frac{|2x|^\a}{ (2x')^\a - |x|^\a }\;(y'-y).
\end{equation} 
By
\eqref{nolog} we have  
$\t\leq \frac{2^\a}{2^\a -1}(y'-y)$ and in particular $\t\leq C_\a  d $. 
Then  we define the points:
\begin{equation} \label{u_2}
\begin{split}
  u_0& =u=(x,y,0),
 \\
  u_1 & = \exp\big( (y'-y) X_2)(u_0) =\big( x,y',  |x|^\alpha (y'-y) \big),
\\
  u_2 & = \exp\Big(\Big(\frac{|x|}{2}-x\Big) X_1\Big)(u_1)
=\Big(\frac{|x|}{2} , y',  |x|^\alpha (y'-
  y) \Big),
\\  u_3  &  = \exp (\tau X_2 ) (u_2)  =\Big( \frac{|x|}{2}, y'+\tau ,
|x|^\a\Big(y'-y+\frac{\t}{2^\a}\Big)  \Big),
\\ u_4 &= \exp\Big(\Big(x'-\frac{|x| }{2}\Big)X_1\Big) (u_3) = \Big(x', y'+\t, 
 |x|^\a\Big(y'-y+\frac{\t}{2^\a}\Big)  \Big),
\\
u_5 & =\exp(-\tau X_2) (u_4) =  \Big(x', y' ,   
|x|^\a\Big(y'-y+\frac{\t}{2^\a} 
\Big) - x'^\a\tau \Big) =(x',y',0) = v.
\end{split}
\end{equation} 
The last identity is due to \eqref{ttaa}.
Let $\gamma_j:[0,T_j] \to  \R ^3_+$ be the   integral curve connecting
 $u_{j-1}$ and $u_{j}$, 
where $T_j>0 $ are such that $\gamma_j(T_j) = u_j$.
The support of $\gamma_j$ is contained in $ \R ^3_+$ for all
$j=1,\ldots,5$.

If $(u,v) \in A$, by  \eqref{piana}  we have  $d(u, v)\simeq
|x-x'|+|y-y'|$. Furthermore, Proposition~\ref{1:prop} and Corollary~\ref{boxxo2}
give $\mu(B(u, d ))\simeq d^{\a+2}$ and so we have  
\begin{equation}\label{kernel} 
  d^{ps}\mu(B(u, d))=  d^{p-1}\mu(B(u, d))\simeq d^{\a+p+1}.        
\end{equation} 
Finally, by \eqref{mix}  we have $d\mu(u)= |x|^\a dx dy$
and $d\mu(v)= |x'|^\a dx' dy'$.

Using these estimates and starting from the inequality
\[
 |f(u) -f(v)| \leq \sum_{j=1}^ 5 |f(u_j) - f(u_{j-1})|,
\]
we obtain
\[
 \int_{A} \frac{|f(u) - f(v)|^p  }{d ^{ps} \mu(B(u,d ))}
\, d \mu(u) \, d \mu(v) \leq C_ {\a, p}  \sum_{j=1}^ 5 I_j,
\]
where
\begin{equation} \label{NU}
 I_j:= \int_A 
\frac{|x|^\a x'^\a }{d^{\a+p+1}} 
\Big(
\int_0^{T_j} |\nablalfa f(\gamma_j(t))| dt \Big)^pdxdydx'dy',\quad j=1,\ldots,5.
\end{equation}

We claim that for all  $j=1,\ldots,5$ we have 
\begin{equation} \label{cillo}
 I_j \leq C_{\alpha,p} \int _{ {\R} ^{3}_+} |\nablalfa f(x,y,z)|^ p dxdydz.
\end{equation}

\step{Estimate  of $I_1$.} The curve connecting $u_0$ and $u_1$ is $\gamma_1(t)=
\mathrm{e}^{tX_2}(x,y,0)$ with $t\in[0,y'-y]$. The corresponding integral is
\begin{equation*}
\begin{aligned}
& I_1 = \int_A 
\frac{|x|^\a x'^\a }{d^{\a+p+1}} 
\Big(
\int_0^{y'-y} |X f(x,y+t, |x|^\a t )| dt \Big)^pdxdydx'dy'.
\end{aligned}
\end{equation*}
Using $x' \leq d/\e_0$ and $0\leq y'-y\leq d$ we obtain  
\[
\begin{split} 
 I_{1 } & 
\leq C_{\alpha }  \int_{\R^2\times\R^2 }\frac{|x|^\alpha }{d^{p+1} } \Big(
\int_0^d
|Xf(x,y+t, |x|^\a t )| dt \Big)^p   dxdy \, dx'dy'
\\
&
\leq C_{\alpha } \int_{\R^2 } \int_0^\infty \int_{\{ d=r \} } \frac{|x|^\alpha 
}{d^{p+1} }
\Big(
\int_0^d
|X f(x,y+t, |x|^\a t)| dt \Big)^p  d\mathcal H^{1}(x',y') \, dr\, 
dxdy 
\\
&
\leq C_{\alpha }  \int_0^\infty  \int_{\R^2 } \Big(  \int_ 0^r
|X  f( x,y+t , |x|^\a t) | dt \Big)^p     |x|^\alpha  
dxdy \, \frac{dr}{r^p}.
\end{split}
\]
We used the coarea formula with $|\nabla d|\simeq 1$.
 Now, by   the Minkowski inequality we obtain 
\[
  I_{1} 
\leq C_{\alpha } \int_0^\infty   \Big(  \frac{1}{r } 
 \int_0^r
 \Big[ \int_{\R^2}  |X f(x, y+t, |x|^\a t) |^p |x|^\a dxdy 
 \Big]^{1/p} dt 
 \Big)^p  dr, 
\]
and  after the change of variable $y\mapsto y+t=\eta$ we can use the  Hardy
inequality to get  
\[ 
  I_{1} 
\leq C_{\alpha,p} \int_0^\infty  \int_{\R^2 } 
|X f(x,\eta , |x|^\a r)|^ p      |x|^\alpha 
dxd\eta  \, dr  
\leq C_{\alpha,p}    
\int_{\R^3_+ } 
|X f(x,y,z )|^ p     
dxdydz, 
\]
by the change of variable $r\mapsto z:= |x|^\a r$.
This proves \eqref{cillo} for $j=1$.

\step{Estimate of  $I_2$.} We connect the points $u_1$ and $u_2 $ using the
integral curve of $X_1$, i.e., the curve  $\gamma_2(t)= 
(x- t\operatorname{sgn}(x),y', |x|^\a (y'-y)) $ with $t\in[0, |x|-x/2]$.
The corresponding integral is
\begin{equation*}
\begin{aligned}
 I_2 &= \int   _A
\frac{|x|^\a x'^\a}{d^{\a+p+1}} 
\Big|
\int_0^{|x| -x/2}  |X  f(x- t\operatorname{sgn}(x),y', |x|^\a (y'-y) )| dt
\Big|^p  dxdydx'dy'
\\ &
\leq C_{\alpha }  \int_{A}\frac{1}{d^{p+1} } \Big( \int_0^{ C d}
|X f(x- t\operatorname{sgn}(x),y', |x|^\a (y'-y))| dt \Big)^p |x|^\alpha   \,
dxdy  dx'dy',
\end{aligned}
\end{equation*}
where   $d  \simeq \max\{  |x-x'|, |y-y'|\}$. We used $x' \leq d/\e_0$  and 
$|x|-x/2\leq C d$. We perform  the  change of
variable in time $ x- t\operatorname{sgn  }x = s$ with  $|s| \leq |x|+|t| \leq C
d$. 
Then, we pass from the variables $(x', y)$ to the variables $\zeta=(\xi,\eta) 
= (x'-x,y'-y)$ with $d\zeta  = dx'  dy$ and $|\z| \simeq d$.
Finally, we use the Minkowski inequality to interchange integration in $ds$ and
$dxdy'$ and we obtain the estimate
\[
 I_{2} \leq C_{\alpha}  \int_{\R^2_+}
\Big(
\int_{|s|\leq C |\zeta| }  \Big[  \int_{A_\z} 
|X_1f(s, y', \eta |x|^\alpha    ))|^ p  |x|^\alpha   \, dx  
 dy' \Big] ^{1/p}  ds \Big)^p   \frac{d\zeta  }{|\zeta| ^{p+1}},
\]
where we let $A_\z = \{ (x,y')
\in\R^2:|x| \leq C|\zeta| \}$. 
By symmetry in the  variable $x$, it suffices to estimate the last integral when
$x>0$.

We perform the change of variable $x\mapsto z =\eta  x ^\alpha  $ with $dz \simeq
\eta   x ^{\alpha-1} dx$, that is equivalent to
\begin{equation} \label{for}
  x ^ \alpha dx \simeq \frac{z^{1/\alpha} }{\eta ^{(\alpha+1)/\alpha}} dz. 
\end{equation}
In order to apply the coarea formula in the $\zeta$ variable for fixed $z$ and
$y'$, we need the following  estimate.

\begin{lemma}
There exists a constant $C>0$ such that for any $r>0$ and   $z>0$   we  have
\begin{equation}
 \label{colla}
  J_{r}(z):=  \int _{D_{r}(z) }
     \frac{ z ^{1/\alpha}  }{\eta ^  {(\alpha+1) /\alpha} } d\mathcal
H^1(\zeta)\leq r C,
\end{equation}
where $D_{r}(z)  = \{ \zeta =(\xi,\eta) \in \R\times\R^+: |\zeta |  = r,\; 0<  z \leq  \eta 
|\zeta |^\alpha   \}$.  
\end{lemma}

\begin{proof} We use  the max-definition $|\zeta|=\max\{|\xi|,
|\eta |\}$.
The estimate is obvious when $D_{r}(z) =\emptyset$. Assume this
is not the case, i.e.,~$0<z<r^{\a+1}$. 
Then by direct calculation 
\begin{equation*}
\begin{aligned}
 J_{r}(z)& = 2\int_{z/r^\a}^r \frac{ z ^{1/\alpha}  }{\eta ^  {(\alpha+1) /\alpha} } d\eta +
 \int_{-r}^r \frac{ z ^{1/\alpha}  }{\eta ^  {(\alpha+1) /\alpha} }  d\xi 
\simeq (  z^{1/\a}r^{-1/\a}+r)\leq C r. 
\end{aligned}
\end{equation*}
 \end{proof}

We  finish  the estimate for $I_{2}$ in the following way.
Let  $E_{r} = \{(y',z) \in\R^2_+: 0<z \leq \eta  r^\alpha \}$.
Using \eqref{for},
the coarea formula, the Minkowski inequality and \eqref{colla}:
\[
\begin{split} 
 I_{2} & \leq C_{\alpha}  \int_{0}^\infty \frac{1}{r ^{p+1} }
\int_{|\zeta|=r}
  \Big(  
\int_{|s|\leq C r }  \Big(  \int_{ E_r  } 
|X f(s, y', z  ))|^ p    \frac{ z ^{1/\alpha}  }{\eta ^  {(\alpha+1) /\alpha} }
dz   
 dy' \Big) ^{1/p}  ds \Big)^p  d\mathcal H^1 (\zeta) dr  
\\ 
& \leq C_{\alpha} 
  \int_0^{\infty } \frac{dr}{r^{p+1}} \Big(
\int_{|s|\leq C r} \Big[   \int_{\R^2_+ } 
\int_{D_r(z)}
|X f(s,y',z)|^p \frac{z^{1/\a}}{\eta^{(\a+1)/\a}}
d\mathcal{H}^1(\z) dy'dz
\Big]^{1/p}ds \Big)^p
\\
&
\leq C_{\alpha}  \int_{0}^\infty \frac{1}{ r^{p+1} } \Big( 
\int_{|s|\leq C r }
  \Big(     \int_{\R^2_+ } 
|X f(s, y',sz  ))|^ p   J_r(z) 
dz   
 dy'  \Big) ^{1/p}  ds \Big)^p   dr  
\\ 
&
\leq C_{\alpha} \int_{0}^\infty   \Big( 
\int_{|s|\leq C r }
  \Big(     \int_{\R^2_+ } 
|X f(s, y', z  ))|^ p   
dz   
 dy'  \Big) ^{1/p}  ds \Big)^p   \frac{dr}{r^p}  
\\ 
&
\leq C_{\alpha,p}      \int_{\R^3_+ } 
|X f(x, y , z  ))|^ p   
 dxdydz     .
\end{split} 
\]
In the last line we used again the    Hardy
inequality.
This proves \eqref{cillo} when $j=2$.

\step{Estimate  of $I_3$.} Let $\gamma_3(t)=
\mathrm{e}^{tX_2}\Big(\frac{|x|}{2}, y' ,
|x|^\a (y'-y)  \Big) $ be the integral curve of $X_2$ connecting $u_2$ and
$u_3$, with  $0\leq t\leq \t$. Recall that the number $\t$ in \eqref{ttaa}
satisfies $\t \leq Cd$.
Also using $0<x'\leq   d/\e_0 $, we obtain
\begin{equation*}
\begin{aligned}
 I_3
   \leq C_\alpha \int _A  \frac{|x|^\a }{d^{ p+1}}
 \Big(\int_0^{Cd} \Big|X  f\Big(\frac{|x|}{2}, y'+t,
|x|^\a\Big(y'-y+\frac{t}{2^\a} \Big) 
\Big)\Big| dt
 \Big)^p dxdydx'dy' .
\end{aligned}\end{equation*}
We perform the change of variable 
$x'\mapsto \xi= x'-x$ and $y\mapsto\eta=y'-y$, so that 
$|\zeta|=|(\xi,\eta)|\simeq d$, and then the change of variable in time
$t\mapsto s=\eta+\frac{t}{2^\a}$, so that $0\leq s\leq C|\z|$.
We get
\begin{equation*}
\begin{aligned}
I_3 &  \leq C_{\alpha} 
 \int_{\R^2}  \int_{\R^2}  
 \Big(\int_0^{C|\z|} \Big|X  f\Big(\frac{|x|}{2}, y'+t,
|x|^\a\Big(\eta+\frac{t}{2^\a} \Big) \Big)\Big| dt
 \Big)^p  |x|^\a  dxdy' \frac{d\z}{|\z|^{p+1}}   
\\&  \leq C_\alpha 
 \int_0^\infty
   \int_{|\z|=r}
   \int_{\R^2}\Big(\int_0^{C|\z|} \big|
Xf (\cdots)
   \big| ds
 \Big)^p
|x|^\a dxdy'
 d\mathcal{H}^1(\z)  
\frac{dr}{r^{p+1}} ,
\end{aligned}\end{equation*}
where
\[
Xf (\cdots) =  X  f\Big(\frac{|x|}{2},
y'+2^\a(s-\eta), 
   |x|^\a s     \Big).
\]

 Next we apply the Minkowski inequality and, after that, we change variable from
$y'$ to $w=y'+2^\a(s-\eta)$. We obtain  
\begin{equation*} 
\begin{aligned}
I_3&  
   \leq C_\alpha 
   \int_0^\infty
   \int_{|\z|=r}
   \Big(\int_0^{Cr }
\Big[ \int _{\R^2}   
\Big|X  f\Big(\frac{|x|}{2}, w, |x|^\a s \Big)\Big|^p
  |x|^\a dxdw \Big]^{\frac 1p}ds \Big)^p
 d\mathcal{H}^1(\z) 
\frac{dr}{r^{p+1}} 
\\&
   \leq C_\alpha 
   \int_0^\infty
   \Big(\int_0^{Cr }
\Big[ \int  _{\R^2}  
\Big|X  f\Big(\frac{|x|}{2}, w, |x|^\a s \Big)\Big|^p
  |x|^\a dxdw \Big]^{\frac 1p}ds \Big)^p \frac{dr}{r^{p }} 
\\&
   \leq C_{\alpha ,p}
\int_0^\infty  \int_{\R^2} 
|X f(|x|/2, w, |x|^\a s) |^p |x|^\a  dx dw dr,
\end{aligned}
\end{equation*}
by the Hardy inequality.

\step{Estimate  of $I_4$.} Let $\gamma_4(t)= \mathrm{e}^{tX_1} 
\Big(\frac{|x|}{2} , y'+\t, |x|^\a
 (y'-y+{\t}/ {2^\a})\Big) $ be the curve connecting $u_3$ and $u_4$, with
$0\leq t\leq x'-\frac{|x|}{2}\leq Cd$. The corresponding integral is
\begin{gather*}
 I_4 = \int _A \frac{|x|^\a x'^\a }{d^{p+\a+1}} 
 \Big(\int_0^{x'-\frac{ |x|}{2}}\Big| Xf \Big(\frac{|x|}{2}+t, y'+\t, |x|^\a
 \Big(y'-y+\frac{\t}{2^\a}\Big)  
 \Big)\Big| dt \Big)^p dxdydx'dy' 
 \\
 \leq C_\alpha 
 \int_A  \frac{|x|^\a }{d^{p+ 1}} 
 \Big(\int_0^{Cd} \Big| X f\Big(s, y'+\t, |x|^\a
 \Big(y'-y+\frac{\t}{2^\a}\Big)   
 \Big)\Big| ds\Big)^p  dxdydx'dy',
 \end{gather*}
where we used $ x' \leq C  d $ and we changed variable $t\mapsto s
=|x|/2+t $ using  the estimate  $0\leq s\leq Cd$.

Next we pass to the variables $y\mapsto\eta =y'-y$ and  $x'\mapsto \xi =x'-x$,
where $\xi\geq |x|-x $ is nonnegative  and observe that
$|\z|:=|(\xi,\eta)|\simeq d$. We use the following rule for changing integration
variables and order 
  \[
\int_{-\infty}^{+\infty}\int_{|x|}^{+\infty} \cdots dx' dx  =
\int_{-\infty}^{+\infty} \int_{|x|-x}^{+\infty} \cdots d\xi dx
=\int_{0}^{+\infty}\int_{- \xi /2}^{+\infty}\cdots dx d\xi .
\]
Recall also that  $|x|\leq Cd\simeq C|\z|$.
Letting $E_{\zeta} = \big \{ (x,y')\in\R^2 : |x|\leq C|\z|,\, x>-\xi/2\big\}$, 
we obtain the following estimate
\[
 \begin{split}
 I_4  & 
 \leq C_\alpha   \int_{\R^ + \times \R^ +}
 \int_{E_\z}
 \Big(\int_0^{C|\z|} \Big| X f\Big(s, y'+\t, |x|^\a
 \Big(\eta + \frac{\t}{2^\a}\Big) 
 \Big)\Big| ds\Big)^p |x|^\a dxdy' \frac{d\z}{|\z|^{p+1}} 
\\
&
\leq C_\alpha 
  \int_{\R^ + \times \R^ +}
   \Big(\int_0^{C|\z|} \Big[ \int_{E_\z }  
 \Big| X  f\Big(s, y'+\t, |x|^\a
 \Big(\eta + \frac{\t}{2^\a}\Big)\Big)\Big|^p
 |x|^\a dxdy'\Big]^{\frac 1p}ds 
  \Big)^p \frac{d\z}{|\z|^{p+1}} 
  \\
&
\leq C_\alpha    \int_{\R^ + \times \R^ +} 
   \Big(\int_0^{C|\z|} \Big[ \int _{E_\z}
 \Big| X f\Big(s, u, |x|^\a
 \Big(\eta + \frac{\t}{2^\a}\Big)\Big)\Big|^p |x|^\a dxdu
 \Big]^{\frac 1p} ds 
  \Big)^p \frac{d\z}{|\z|^{p+1}} .
\end{split}
\]
In the change of variable  $y'\mapsto u=y'+\t$ we used the fact that $\t  $ is
independent of $y'$ after letting $\eta = y'-y$,   by~\eqref{ttaa}.
The next step is the change of variable
 \begin{equation}\label{joao} 
 z=|x|^\a\Big(\eta+\frac{\t}{2^\a}\Big) = \frac{2^\a(x+\xi)^\a |x|^\a}{2^\a(x+\xi)^\a
-|x|^\a}\eta.
\end{equation}
Observe  that $|x|^\a\eta \leq z\leq  \frac{2^\a}{2^\a -1} |x|^\a
\eta$, for all $\xi>0$ and    $x\in\left]-\xi/2,+\infty\right[$.
Note that if   $x$   gets too close  to $  - \frac 23\xi$, then the estimate  fails.
Furthermore, we have 
\begin{equation*}
 \frac{dz}{dx}=\a 2^{\a}\eta\frac{(x+\xi)^{\a-1}|
 x|^{\a-1}}{\big(
 2^\a(x+\xi)^\a -|x|^\a
 \big)^2}\big( -|x|^{\a+1} +2^\a(x+\xi)^{\a+1}
 \operatorname{sign}(x)\big) \simeq \eta |x|
 ^{\a-1}\operatorname{sign}(x).
\end{equation*}

To proceed, we split the integration on $E_\z$ into the integration on the
following two sets
\[
 E_\z^+ =  \big \{ (x,y')\in\R^2 : 0\leq  x \leq C|\z|\big\}
\quad\textrm{and}\quad
E_\z^- =  \big \{ (x,y')\in\R^2 :-\xi/2<x<0\big\}.
\]
We denote the corresponding integrals $I_4^+$ and $I_4^-$, 
respectively. 

We estimate $I_4^+$. 
With the   
 change of variable \eqref{joao}, by  the previous discussion we get
  $ x^\a dx \simeq \frac{z^{1/\a}}{\eta^{(\a+1)/\a}} dz $
and  thus
\begin{equation*}
\begin{aligned}
 I_4^+ &  \leq C_\alpha 
  \int_{\R^{+}\times\R^+ } 
   \Big(\int_0^{C|\z|}\Big[ \int_{  0<z < C \eta |\z|^\alpha}
 | X f (s, u, z  ) ) |^p  \frac{z^{1/\a}}{\eta^{(\alpha+1)/\a}} dzdu
 \Big]^{\frac 1p}  ds 
  \Big)^p\frac{d\z}{|\z|^{p+1}} 
\end{aligned}
\end{equation*}
and we conclude using~\eqref{colla}. 
The estimate of $I_4^-$ is analogous.

\step{Estimate of $I_5$.}
The curve $\gamma_5$ connecting $u_4$ to $u_5=v$, in a backward 
parametrization, gives the following   estimate for the  integral $I_5$:
\begin{equation*}
\begin{aligned}
 I_5 \leq C_\alpha  \int_A \Big( 
 \int_0^\t |X f(x',y'+t,x'^\a t )|dt  \Big)^p \frac{|x|^\a x'^\a
}{d^{\a+p+1}}dxdydx'dy',
\end{aligned}
\end{equation*}
and using  $\t \leq C d$ the evaluation of this  integral is  identical to
the one for~$I_1$.

\begin{remark}
\label{remo}
In this section, we proved the integral estimate \eqref{equa}
starting from a couple of points $u = (x,y,0)$ and $v=(x',y',0)$ 
satisfying \eqref{yy'} and \eqref{xx}. Since \eqref{yy'} can be always assumed,
we briefly discuss the case when \eqref{xx} fails. This can happen in two
situations: either $y'>y$ and $x>x'>0$, or $y'>y$ and $x<-x'<0$.  

In the first case, it suffices to add to
the points $u=(x,y,0)$ and $v=u'=(x',y',0)$ a third point $u''=(x'', y'',
0)=:(2x-x', 2y'-y)$. Both the ordered pairs of points $u, u''$ and
$u', u''$
satisfy \eqref{xx} and \eqref{yy'}. Then it suffices to use the triangle
inequality 
and to recognize that the kernels appearing in the Besov seminorm related to the
three 
pairs  of points $(u,u')$, $(u, u'')$ and $(u', u'')$ are mutally equivalent.

In the second case, we add a third point $u''=(x'',y''):=(-x, 2y'-y)  $. Both
the ordered pairs $(u, u')$ and $(u, u'')$ satisfy \eqref{yy'} and \eqref{xx},
and again, since 
$d(u, u')\simeq d(u, u'')\simeq d(u', u'')$ and 
the kernels appearing in the Besov norm related to the three 
pairs  of points $(u,u')$, $(u, u'')$ and $(u', u'')$ are all equivalent.

\end{remark}
 
\section{Noncharacteristic case}

\label{pop}
 
We are in the case $d\leq  \e_0 |x| $ and $d\leq \e_0|x'|$,
where the constant $\e_0>0$ will be fixed along the proof.  
Since $|x-x'|\leq d$ we can assume   that $x',x
\geq  d / \e_0 $, i.e., they are   both positive. 
Without loss of generality, we  assume that
\begin{equation}\label{vilo} 
 y<y'\qquad \text{ and} \qquad x>x'>0.
\end{equation}
The case $y<y'$ and $0< x < x'$  is discussed  in Remark \ref{remarco}. If   
$x,x'$ are both negative, it suffices to apply the transformation 
$(x,y,t)\mapsto (-x,y,t)$, see~\eqref{cetra}.

With the notation $u = (x,y,0)$, $v=(x',y',0)$ and $d = d(u,v)$, we consider 
the
integration domain  
\begin{equation}\label{perlamiseria} 
 \begin{split}
  B  & = \{(u,v)\in\R^2\times \R^2 : y'\geq y,\, x > x'\geq  d /\e_0 \}.
\end{split}
\end{equation}  
Notice that for  $\e_0$  sufficiently small we may also assume that $x\simeq
x'$.

Starting from $u$, we introduce certain intermediate points interpolating $u $
and $v$. Consecutive points are connected by integral curves of the vector
fields $\pm X_2$ and $\pm Z$ with $Z = X_1+X_2 $.

Let $\sigma = \sigma(u,v) >0$ be the number
$
 \sigma :  =y'-y+ x-x' .
$
We define the points 
\[
\begin{split}
u_0 &= u = (x,y,0)
\\
   u_1 & = \exp(\sigma X_2)(u_0) = (x,y'+ x-x', x^\a \sigma)
\\
 u_2 & = \exp\big( (x'-x) Z\big) (u_1) = \left(x'   , y' , \s x^\a  +
\frac{ {x' }^{\a+1}- x^{\a+1}  }{\a+1} \right).
  \end{split}
\]
Notice that $u_2 \in \R^3_+$, because
\begin{equation}\label{birillino} 
\begin{aligned}
z'  
 : =\sigma  {x}^\alpha -\frac{  {x}^{\alpha+1}   -x'^{\alpha+1} }{
 \alpha+1 }
& =(y'-y)x^\a +\int_{x'}^x (x^\a - t^\a)dt> 0, 
\end{aligned}
\end{equation}  
as soon as  $y'>y$ or $ x>x'>0$.
This  inequality may fail if  $x< x'$, but
see Remark~\ref{remarco}. Observe also that
\begin{equation}
 \label{zetaprimo}
 z'\leq \sigma x^\a \leq  C d \;  x^{\a }\leq  C  \e_0 x'^{\a+1},
\end{equation}
because  $ x' \geq d /\e_0$ and $x\simeq x'$.

To reach $v$ starting from
$u_2$ we follow for a positive time  an approximation of the commutator
$
  [X_2,Z ] = [ X_2, X_1 + X_2]  = - \alpha x^{\alpha -1}
\frac{\partial}{\partial z}.
$
In a standard way, we approximate  the flow along
this commutator with a composition of
flows of the vector fields $\pm Z$ and $\pm X_2$.

Let  $\tau= \tau(u,v)$ be the positive solution  of the   equation 
the
equation $ {z'}+\tau   {x'}
^\alpha-\tau ( {x'}+\tau)^{\alpha} =0$, that reads
\begin{equation} 
\label{tautau}
\begin{aligned}
\tau ( {x'}+\tau)^{\alpha} - \tau   {x'}
^\alpha  
   =   \big(y'-y+x-x'\big){x}^\alpha -\frac{{x}^{\alpha+1}  -x'^{\alpha+1} }{
\alpha+1}.
\end{aligned}
\end{equation}
This equation has a
unique positive solution $\tau\geq 0$.
By \eqref{dasotto} we have 
\begin{equation} 
\label{dillo}
 \tau \simeq  \min \Big\{ \sqrt{\frac{z'}{  {x'} ^{\alpha-1} }},
 z'^{1/(\a+1)}\Big\}
 \simeq
d((x',y',z'),(x',y',0)) \leq C
d(u,v) ,
\end{equation}
 where we used  Theorem~\ref{ddss}  
  and  the triangle
inequality.

Finally, we define the following further points:
 \[
\begin{split}
u_3 & = \exp(\t X_2) (u_2) 
   = \big( {x'}, {y'}+\tau , {z'} + x'^\a \t \big)
\\
u_4 & 
= \exp(\t Z)(u_3) = 
\bigg(
{x'} + \tau, {y'}+2\tau , {z'}
+x'^\a \tau +\frac{  (x'+\t)^{\a+1}- x'^{\a+1}  }{\a+1}
\bigg )
\\
u_5 & =\exp(-\t X_2) (u_4) = \bigg( {x'} + \tau,  {y'}+ \tau
, {z'}  
+x'^\a \tau +\frac{ (x'+\t)^{\a+1}- x'^{\a+1}  }{\a+1}
-(x'+\t)^\a\tau
 \bigg)
  \\
u_6 & = \exp(-\t Z) (u_5) = \big( {x'}  , {y'} 
, {z'} + x'^\a\tau -(x'+\t)^\a\tau \big) =(x',y',0).
\end{split} 
\]
In the last identity we used \eqref{tautau}.
For $i=1,\ldots,6$, we denote by $\gamma_i:[0,T_i]\to\R^3_+$ the curve
connecting $u_{i-1}$ and $u_i$, where $T_i$ is such that $\gamma_i(T_i) = u_i$.

According to Proposition \ref{1:prop}
and Corollary \ref{boxxo2},  for points $(u,v) \in B$ the kernel in \eqref{3:eq}
satisfies
\begin{equation} \label{AA}
 d^{ps} \mu(B(u, d ))\simeq d^{p+2} x^{\alpha-1},
\end{equation} 
and by \eqref{piana}, the distance function has the structure
\[
d \simeq  |x-x'| + |y-y'| + \sqrt{|y-y'|  x } \simeq \max \{ |x-x'|,
\sqrt{|y-y'|  x } \}. 
\] 
The last equivalence follows from $0\leq |y'-y|\leq \sqrt{|y'-y|\,d}\leq
\sqrt{\e_0}  \sqrt{|y'-y|  \, |x| }$.

By the triangle inequality we obtain  
\begin{equation}  
\int_{B} \frac{|f(u) - f(v)|^p}{d^{ps} \mu(B(u,d))} \, d \mu(u) \, d
\mu(v) \le  C_{\alpha, p} \sum_{i=1}^{6} J_{i},
\label{3:eq} 
\end{equation}
where 
\[
J_{i}  :=
\int_{B}  \Big( \int _0^{T_i}  |Xf(\gamma_i(t) |dt \Big) ^p
\frac{x^\alpha  \, x'^\alpha }{d^{p+2} x^ {\alpha-1}  } dxdydx'dy',  \qquad
\text{for } i=1,\ldots,6.
\]
We claim that the integrals $J_i$ satisfy
\[
J_i \leq C_{\alpha,p} \int_{\R^3_+}  
|\nablalfa  f(x,y, z) )|  ^p
dx dy dz. 
\]

\step{Estimate of $J_{1}$.}
Starting from the point $u_0 = (x,y,0)$, we follow  the vector field $X_2$
for a  positive time $ \sigma = y'-y+ x-x' \leq d$. 
Using    the  estimate in \eqref{AA}, we arrive at the inequality
\begin{equation}
 \label{THET}
 J_{ 1} \leq  \int_{B} \frac{1}{d^{p+2}
x^{\alpha-1} } \Big( \int_0^{d} |X f( 
x,y+t, t x^\alpha  )| \, dt \Big)^p
x^\alpha  {x'}^\alpha
 dxdy dx' dy'.
\end{equation}  
We use the coarea formula along with the following lemma.

 \begin{lemma} There exists a constant $C>0$ such that for any $x,y \in\R$ with
$x\geq    r/\e_0>0$
we have 
\begin{equation} \label{gillo}
 \int_{D_r(x,y)}\frac{|x'|^{\alpha } }{|\nabla d(x',y')|} d\mathcal
H^1(x',y') \leq   C r^2 x ^{\alpha-1}  , 
\end{equation}
where $D_r(x,y) = \{(x',y')\in\R^2: d =r\}$ with $d  =   \max \{
|x-x'|,
\sqrt{|y-y'|  x } \}$.
\end{lemma}

\begin{proof}
The set $D_r(x,y)$ is the boundary of the rectangle 
$[x-r, x+r]\times [y-\frac {r^2}{x},y+\frac {r^2}{x}]$. When  in the max-definition of $d$ we have $d = \sqrt{|y-y'| x}$,
then, on $\{d=r\}$, the gradient   of $d$   satisfies
$
 |\nabla d(x',y') | =  \frac{\sqrt{x}}{2 \sqrt{|y-y'|}} = \frac{x}{2r}. 
$
In the corresponding part of the integral \eqref{gillo}, the function  
$|x'|^{\alpha}$ is integrated on the interval $(x-r,x+r)$.

On the set where  $d=|x-x'|$ we have $|\nabla d|=1$ and, in \eqref{gillo}, the 
constant $|x'|^\alpha = |x\pm r|^ \alpha \simeq |x|^\a=x^\a $ 
is integrated for $y'\in (y- r^2/x,y+r^2/x)$. In both
cases the claim follows. 
\end{proof}

Starting from  \eqref{THET}, by  the coarea formula and
inequality~\eqref{gillo}, by the Minkowski and Hardy inequalities we obtain 
\begin{equation*}
\begin{split}
 J_1 &  
   \leq \int_{x>0} \int_0^{\e_0 x}\int
_{D_r(x,y)}\frac{x'^\a}{|\nabla d |}
   \Big( \int_0^{d} |X f(x, y+t, x^\a t )| \, dt \Big)^p d\mathcal
H^1(x',y')  \frac{dr}{r^{p+2}} xdxdy 
\\& 
   \leq C_\alpha  \int_0^{\infty}\int_{x>0}
   \Big( \int_0^{r} |X f(x,y+t, x^\a t )| \, dt \Big)^p   x^\a dxdy
\frac{dr}{r^{p }}  
 \\    
&   \leq C_\alpha 
\int_0^{\infty} \Big(
   \int_0^{Cr}    \Big[  \int _{x>0}  |X f(x,y+t,x^\a t )|^p  x^\a dxdy 
\Big]^{1/p}dt
\Big)^p \frac{dr}{r^{p }} \\
&
 \leq C_{\alpha,p} 
  \int_0^\infty   \int_{x>0}
|X f(x,y+r,  rx^\alpha )|  ^p
x^\alpha    dx   dy dr 
\\
&
 \leq C_{\alpha,p}   \int_{\R^3_+}  
|X  f(x,y, z) )|  ^p
dx dy dz,
\end{split} 
\end{equation*}
as required.

\step{Estimate of $J_{2}$.} The integral curve connecting $u_1$ and $u_2$ is  
\[
 \gamma_2(t) = \Big ( x- t ,y + \sigma - t , \sigma x^\alpha 
-\frac{x^{\alpha+1} -
(x- t)^{\alpha+1}}{\alpha+1 }  
\Big)  ,\qquad\text{for } 0\leq t \leq  x-x' .
\]
where $\sigma = y'-y+x-x'$. Using \eqref{AA} and $0\leq x-x' \leq d$, we start
from the estimate
\[
 J_{2 } \leq   \int_B  \Big(
   \int_0^{x-x'} | Xf(\gamma_2(t))| 
   dt \Big)^p   \frac{x^\alpha {x'}^ \alpha dxdy dx'dy'}{d^{p+2}
x^{\alpha-1} }  .
\]
We perform the change of variable from  $x',y'$ to $h=(h_1, h_2)$
\begin{equation}\label{cidivu} 
 h = (h_1,h_2) 
     = (x-x',\sqrt{(y'-y)x})\in\left]0,+\infty\right[\times
     \left]0,+\infty\right[.
\end{equation}   
Note that   $|h|\simeq d$. The  Jacobian satisfies 
$dx'dy' \leq C\frac{|h|}{x} dh$ and so we obtain  
\[
 J_{2 } \leq C 
 \int_{ \wh B } \frac{1}{|h| ^{p+1} } \Big(
   \int_0^{h_1}  \Big| Xf\Big( x-t  , y+\wh \sigma -t , \wh \sigma x^\alpha
-   \frac{  x^{\alpha+1} -  (x-t)^{\alpha+1} ) }{\alpha+1}
  \Big| dt
\Big)^p    {x} ^ \alpha   dxdy dh,
\]
where $\wh B = \{(x,y,h) \in\R^4: h_1, h_2>0,\quad  x \geq   |h|/\e_0 \}$ and
$ \wh\sigma : =
h_2^2/x+ h_1 \leq C|h|$.  
Now we perform the change of variable in time
\[
\begin{aligned}
  s  & =\phi_{x,h} (t) = \frac{ \wh \sigma x^\alpha
+((x-t)^{\alpha+1}-x^{\alpha+1})/(\alpha+1)}{(x- t) ^\alpha} .
\end{aligned}
\]
By \eqref{birillino}, on the integration set we have $s\geq 0$.
Moreover, it is $0\leq t \leq C|h|$ and $x\geq  d/\e_0\simeq  |h|/\e_0$, 
with constants independent of $\e_0$.   Then, choosing $\e_0>0$ small enough, we
have  
$x-  t \geq  
x/2. $
  Since $\wh \sigma =\frac{h_2^2}{x}+|h_1|\leq C  d\simeq |h|$, we conclude that
 $0\leq s\leq C|h|$. An easy computation also shows that, for
$0\leq t \leq C|h|$,
\begin{equation} 
\label{fifa}
 |\phi'_{x,h} (t)+1| \leq  C \frac{|h|}{|x|} \leq \frac 12,
\end{equation}
for  $|x| \geq  |h|/\e_0$ and 
$\e_0$ small enough. 
Letting $\widehat t = \phi_{x,h}^{-1}(s)\in[0,h_1]$  we arrive at the estimate
\[
 J_{2 } \leq C_\alpha \int_{ \wh B }\Big(
   \int_0^{C |h|}  | Xf(x- \wh t, y+\wh \sigma  -\wh t
   , s (x-\wh t)^\a 
)| ds
\Big)^p    x ^ \alpha   dxdy  \frac{dh}{|h| ^{p+1} } .
\]
Next we use  the  Minkowski inequality to interchange integration in $ds$ with
integration in $dxdy$: 
\[
 J_{2 } \leq C_\alpha    \int_{ \R^2  } \Big(
   \int_0^{C |h|}  \Big[  \int_{\{x\geq |h|/\e_0 \} } | Xf(x-\widehat t, y+\wh
\sigma -\widehat
t , s (x- \widehat t) ^\alpha)|^ p  x ^\alpha   dxdy \Big]^{1/p}   ds
\Big)^p       \frac{dh}{|h| ^{p+1} }.
\]
The change of variable $\overline{y} =  y+\wh \sigma -\widehat
t$ in the inner integral is elementary because 
$\wh \sigma$ and $\widehat
t$ do not depend on $y$. Let us consider the transformation $x\mapsto \ol x$ defined by    
\begin{equation}\label{elk} 
 \overline{x} = x-\wh t =x- \phi_{x,h}^{-1}(s) .                                                
 \end{equation} 
Note first that $\ol x\in [x-h_1, x]$, because $s\in \phi_{x,h}([0, h_1])$.
Using the  definition of $\phi_{x,h}$ and $\wh\s$,  
we see that  \eqref{elk} can be  written in the form 
\begin{equation}\label{riscriviamo}   
F_s(\ol x):= \frac{\ol x^{\a+1}}{\a+1}- s\ol x^\a =\frac{x^{\a+1}}{\a+1}-h_2^2 x^{\a-1}-h_1 x^\a
=: G_h(x).
\end{equation}
It is easy to see by one-variable calculus  that
$F_s:\left[(\a+1)s,+\infty\right[\to \left [0, +\infty\right[$ is a strictly
increasing bijection with strictly positive derivative. Furthermore,   if $\e_0$
is 
small enough  then 
\[
G_h:\left[|h|/\e_0,+\infty\right[\to G_h(\left[|h|/\e_0,+\infty\right[)=:I_h\subseteq \left[0,+\infty\right[ 
\]
  satisfies  $G_h'(x)>0$ for all $x>|h|/\e_0$.  Then \eqref{elk} can be
   written as a true change of variable $x= G_h^{-1}(F_s(\ol x))$,  where  
$\ol x\in F_s^{-1}(I_h)\subset 
  \left[0,+\infty\right[$  and  by \eqref{riscriviamo} we  have the following change in the
integration element 
\begin{equation}
\label{girino}      (  \overline{x}^\alpha-
\alpha s\overline{x}^{\alpha-1} )
d\overline{x} = (x^\alpha- (\alpha -1) h_2^2 x^{\a-2}-\a h_1 x^{\alpha-1} ) dx \simeq x^\alpha
dx.
\end{equation}  
 Then, $x^\a dx\leq C_\alpha \ol x^\a d\ol x$ and 
 the estimate can be finished by the coarea formula and the Hardy inequality as follows:
 \[
\begin{split}
J_{2 } & \leq C_\alpha \int_{ \R^2  } \frac{1}{|h| ^{p+1} } \Big(
   \int_0^{C |h|}  \Big(  \int_{\ol x>0} | Xf(\overline x,\overline{y},   s
\overline{x}^\alpha )|^ p  {\overline x} ^ \alpha    d\overline{x}d\overline{y}
\Big)^{1/p} ds
\Big)^p      dh
\\
&
\leq C_\alpha     \int_{0 }^\infty     \Big(
   \int_0^{C r}  \Big(  \int_{\ol x>0} | Xf(\overline x,\overline{y},   s
\overline{x}^\alpha )|^ p  {\overline x} ^ \alpha    d\overline{x}d\overline{y}
\Big)^{1/p} ds
\Big)^p      \frac{dr}{r^p}
\\
&
\leq C_{\alpha,p}  \int_{0 }^\infty     \int_{\ol x >0} | Xf(\overline
x,\overline{y},  r
\overline{x}^\alpha)|^ p  {\overline x} ^\alpha   d\overline{x}d\overline{y}
dr 
  \\
  &
\leq C_{\alpha,p} 
     \int_{\R^3_+} | Xf(x,y,z)|^ p  
dxdydz.
\end{split}
\]

\step{Estimate of $J_{3}$.} The curve connecting $u_2$ and $u_3$ is
$
 \gamma_3( t) =  (x',y'+t ,z' +t {x'}^\alpha)$, where 
$ t\in[0,\tau]$ and 
$\tau  $ solution of  \eqref{tautau}. The quantity $z'$ is defined in
\eqref{birillino}.  
Using \eqref{AA} and \eqref{dillo}, we can start from the estimate 
\begin{equation}
 \label{THET3}
 J_{ 3} \leq \int_{B} \frac{1}{d^{p+2}
x^{\alpha-1}} \Big( \int_0^{C d} |X f(x',y'+t ,z' +t {x'}^\alpha)| \, dt
\Big)^p
x^\alpha  {x'}^\alpha
\, dx dy  dx'  dy'.
\end{equation}
Observe that
$
{z'}\leq   \s x^\a \leq C { x'^\alpha}d .$ 
So, the change of variable in time $z' +t {x'}^\alpha = s {x'}^\alpha$ gives $dt
=ds$
and the integration set in $s$ is contained in $  [0,  Cd]$. Then we get
\begin{equation*}
 J_{ 3} \leq C_{\alpha } \int_{B} \frac{1}{d^{p+2}
x^{\alpha-1} } \Big( \int_0^{C d} |X f(x',y'+\widehat t, s {x'}^\alpha ) )| \,
ds \Big)^p
x^\alpha  {x'}^{\alpha} 
\, dx dy  dx'  dy',
\end{equation*} 
where $\wh t   = s -   z'/{x'}^\alpha $.

Next we change variables from   $(x,y)$ to $  h = (h_1,h_2)$ as in \eqref{cidivu}
with Jacobian  $dh = \frac{x}{2 h_2} dxdy$,
and so we obtain $xdxdy=|2h_2| dh\leq C |h| dh$. Therefore
\begin{equation}
 \label{THET44}
 J_{ 3} \leq C_{\alpha } \int_{\R^2  } \frac{1}{|h| ^{p+1}
} \Big( \int_0^{C |h| } |X f(x',y'+\widehat t, s {x'}^\alpha ) )| \, ds
\Big)^p
 { x'}^\alpha
\,   dx'  dy'\, dh,
\end{equation} 
where $\wh t   = s -  z'/{x'}^\alpha  =s-\frac{1}{x'^\a}\Big((\frac{h_2^2}{x'+h_1}+h_1)
(x'+h_1)^\a - 
\frac{( x'+h_1)^{\a+1}-x'^{\a+1})}{\a+1} \Big) $   does not depend on $y'$.

We use the Minkowski inequality to interchange integration in $ds$ and $dx'dy'$:
\begin{equation*}
 J_{ 3} \leq C_{\alpha }  \int_{\R^2} \frac{1}{|h| ^{p+1}
} \left( \int_0^{C |h| }  \Big( \int_{x'>0 } |X f(x',y'+\widehat t, s
{x'}^\alpha ) )| ^p   {x'}^\alpha dx' dy' \Big)^{1/p}   ds
\right)^p
 dh.
\end{equation*} 

The change of variable $\bar y = y'+\widehat t$
satisfies   $ d\bar y = d y'$, and we finally obtain
\begin{equation*}
 J_{ 3} \leq C_{\alpha }  \int_{\R^2} \frac{1}{|h| ^{p+1}
} \left( \int_0^{C |h| }  \Big( \int_{x'>0} |X f(x',\bar y , s
{x'}^\alpha ) )| ^p   {x'}^\alpha dx' d\bar y \Big)^{1/p}   ds
\right)^p
 dh.
\end{equation*} 
Ultimately, we conclude using the coarea formula and the Hardy
inequality.

\step{Estimate of $J_{4}$.}
The curve connecting $u_3$ and $u_4$ is  
\[
\begin{aligned}
\gamma_4(t) & =  \big(x'+t , y'+\tau+ t , z' + \tau {x'}^\alpha
+[(x'+t)^{\alpha+1}-{x'}^{\alpha+1}]/(\alpha+1)\big)
,\quad t\in[0, \t].
\end{aligned}
\]
Using \eqref{AA} and \eqref{dillo}, we can start from the estimate 
\begin{equation}
 \label{THET9}
J_{ 4} \leq  \int_{B} \Big( \int_0^{C d} |Xf(\gamma_4 (t) )| \, dt \Big)^p
\frac{x^\alpha x'^{\alpha}
\, dx dy  dx'  dy'}{d^{p+2}
x^{\alpha-1} }
.
\end{equation} 
With the change of variables \eqref{cidivu} from  variables $(x,y)$ to $h=(h_1,h_2)$,  we obtain  
\[
J_{  4} \leq C_\alpha   \int _{ \{ x '\geq   |h|/\e_0 \} }
\Big(
   \int _0^{C |h|}   | Xf(\cdots)  | dt
\Big)^p   \frac{ {x'} ^
\alpha  dx'dy' dh}{|h| ^{p+1} } ,
\]
where 
\[
(\cdots)=
 \Big(x'+t, y'+\wh \tau+ t  , \wh {z'}+ \wh
\tau {x'}^\alpha
+\frac{ (x'+t)^{\alpha+1}-{x'}^{\alpha+1}}{ \alpha+1}
\Big),
\]
and
\begin{equation}\label{zetacap} 
 \wh z'=\wh z'(x',h)=  h_2^2(h_1+x')^{\a-1} +h_1 
 (h_1+x')^\a -\frac{(x'+h_1)^{\a+1}-x'^{\a+1}}{\a+1}.
\end{equation}
Notice that  the unique solution $\wh \tau=\wh \tau(x',h)$ of
$\t((x'+\t)^\a-x'^\a)= \wh z'$ does not depend on $y'$. 

In the next step, we perform the change of variable in time  
\begin{equation}
\label{caio} 
\begin{aligned}
 s &=\wh\phi_{x',h}(t):  = \frac{\wh z'(x',h)+ \wh \tau (x',h){x'}^\alpha 
+[(x'+t)^{\alpha+1} -{x'}^{\alpha+1} ]/(\alpha+1) }{(x'+t)^\alpha}  
\\&
 =  \frac{x'+t}{\a+1} +(x'+t)^{-\alpha}
 \Big\{\wh z'(x',h)+ \wh \tau (x',h){x'}^\alpha 
 -\frac{{x'}^{\alpha+1} }{\alpha+1}   \Big\}
\end{aligned}
\end{equation} By \eqref{zetaprimo}, \eqref{tautau} and the noncharacteristic   case, we have 
$
0<\wh z'(x',h)+\wh \tau(x',h)x'^\a
\leq C d   x'^{\a}\leq C \e_0   x'^{\a+1}$. Furthermore, an easy
computation furnishes
$\frac {1}{2(\a+1)}\leq  \phi'(t) \leq   \frac{2}{\a+1}$ for all $0\leq t\leq
C|h|$, if $\e_0$ is small enough. 
Therefore, $\phi_{x',h}:[0, C|h|]\to \phi_{x',h}([0, C|h|])=:A_{x' ,h}\subset
[0, \wh C|h|] $ is a monotone increasing bijection.
Letting $\wh t = \wh \phi_{x',h}^{-1}(s)$, we obtain 
\[
J_{ 4} \leq C_\alpha  \int_{ \{ x '\geq   |h|/\e_0 \} } \frac{1}{|h| ^{p+1} }
\Big(
   \int_0^{\wh C|h|}  | Xf(x'+\wh t, y'+\wh \tau+ \wh t  , s(x'+\wh t)^\alpha )
| ds
\Big)^p    {x'} ^ \alpha    dx'dy' dh.
\]
We use the Minkowski inequality to interchange integration in $ds$ and 
$dx'dy'$:
\[
 J_{ 4} \leq C_\alpha \int_{\R^2  }  \Big(
   \int_0^{\wh C |h|}  \Big[ \int _{ \{ x '\geq   |h|/\e_0 \} } | Xf(
   x'+\wh t, y'+\wh \tau+ \wh t  , s(x'+\wh t)^\alpha 
) |^p  
{x'} ^ \alpha   dx'dy'\Big] ^{1/p}   ds
\Big)^p    \frac{dh}{|h| ^{p+1} }.
\]
The functions $\wh \tau$ and $\wh t $ do not depend on $y'$.
So the change of variable $\bar y =  y'+\wh
\tau+ \wh t $ is a   translation and  $d\bar y = d y'$.

Next we look at the  transformation 
$\bar x=x' +\wh t =x'+\wh\phi_{x',h}^{-1}(s) $, 
where we know that $\wh t\in [0, \wh \tau (x', h) ]\subset [0, C |h|]$. 
Such transformation is equivalent  to $\wh\phi_{x',h} (\ol x- x')= s$. Since 
the explicit form of \eqref{caio} gives
\begin{equation*}
\begin{aligned}
  \wh\phi_{x',h}(t)   =\frac{1}{(x'+t)^\a}
 \bigg\{
& h_2^2 (x'+h_1)^{\a-1}+ h_1 (x'+h_1)^\a 
 \\ 
 &+\frac{(x'+t)^{\a+1}- (x'+h_1)^{\a+1}}{\a+1} +\wh\t(x',h)x'^\a
 \bigg\},
\end{aligned}
\end{equation*}
 the transformation can be  written as 
\begin{equation}\label{cipolla} 
\begin{aligned}
 F_s(\ol x):  = \frac{\ol x^{\a+1}}{\a+1}-s\ol x^\a
 & = \frac{(x'+h_1)^{\a+1}}{\a+1}-h_2^2 (x'+h_1)^{\a-1}
\\& \qquad -  h_1(x'+h_1)^\a -\wh \t (x',h) x'^\a
  =: \wh G_h(x').
\end{aligned}
\end{equation} 
We are in a situation similar to \eqref{riscriviamo} in the estimate of $J_2$,
but here the right-hand side is slightly more complicated. As in the previous case,  
we see by one-variable calculus that $F_s:\left[(\a+1)s,+\infty\right[\to \left [0, +\infty\right[$ is a strictly increasing bijection with strictly positive derivative. Concerning the right-hand side, it suffices to show that  
\[
\wh G_h:\left[|h|/\e_0,+\infty\right[\to \wh G_h(\left[|h|/\e_0,+\infty\right[)=:\wh I_h\subseteq \left[0,+\infty\right[ 
\]
  satisfies  $\frac{d}{dx'}\wh G_h (x')>0$ for all $x'>|h|/\e_0$. 
All terms are similar to those appearing in $J_2$,  but here we need to show the
following  further inequality:
\begin{lemma}\label{diciamo} 
We have the estimate 
\begin{equation*}
 \label{latau}
 \Big|\frac{\p}{\p x'}x'^{\a }\wh \tau(x',h)  \Big| \leq \s_0 x'^\a, \quad\text{for all $x'>|h|/\e_0$,}
\end{equation*} 
where the constant $\s_0$ can be made small by choosing  $\e_0$ small enough. 
\end{lemma}

The proof of the claim is postponed   after the end of the estimate of~$J_4$.
To conclude the estimate of $J_4$, as a consequence of Lemma \ref{diciamo}, we
discover that  we may write
$x'= \wh G_h^{-1} F_s(\ol x)$ and the change of  variable has strictly positive
derivative, the variable   $\ol x $ is nonnegative  and  differentiating
\eqref{cipolla}, we get $x'^\a dx'\leq C \ol x^\a d\ol x$.  

 Ultimately, we obtain the estimate 
\[
 J_{ 4} \leq C_{\alpha }  \int_{\R^2  } \frac{1}{|h| ^{p+1} } \Big(
   \int_0^{C |h|}  \Big( \int _{\ol x >0} | Xf(\bar x, \bar y  ,
s{\bar x} ^\alpha )|^p   {\bar x} ^ \alpha     d\bar x d \bar y \Big) ^{1/p}  
ds
\Big)^p    dh,
\]
and the argument is concluded in the usual way.

\begin{proof}[Proof of Lemma \ref{diciamo}]
To prove claim  \ref{diciamo}, we  first   get an explicit form of $\wh \tau$. Starting from  
\begin{equation}
\label{cuccurullo} 
 \t((x'+\t)^{\a }- x'^{\a })=\wh z':= h_2^2(x'+ h_1)^{\a-1}+ h_1 (x'+ h_1)^\a -\frac{(x'+ h_1)^{\a+1}-x'^{\a+1}}{\a+1}
\end{equation}
and letting $v(s)= s((1+s)^\a -1)$, 
we see that $\frac{\t}{x'} =v^{-1}(\frac{z'}{x'^{\a+1}})$. Recall
that the ratio $\frac{z'}{x'^{\a+1}}$ is close to zero 
if $\e_0$ is small (see~\eqref{zetaprimo}). Furthermore, we have $v(s)\simeq
s^2$ and $v'(s)\simeq s$ for $s$ close to $0$. 
So \eqref{latau} is equivalent to 
\begin{equation*}
\begin{aligned}
&\Big |\frac{\p}{\p x'}
\Big (x'^{\a+1} v^{-1}\Big(\frac{z'}{x'^{\a+1}}\Big)\Big)\Big|\leq \s_0 x'^\a                                                                                                                                                                          
 \quad
\Leftrightarrow 
\quad 
\Big| (\a+1)\frac{\t}{x'}+ \frac{x'^{\a+1}}{v'(\t/x')} \frac{\p}{\p
x'}\frac{z'}{x'^{\a+1}}\Big|<\s_0  .
\end{aligned}
\end{equation*}
 The first term is easily controlled. In order to control the second one,
observe
that   
 \begin{equation}
 \label{girini} 
  \frac{x'}{v'( \t/x')}\simeq \frac{x'}{\t/x'}=\t \frac{x'^2}{\t^2}\simeq \frac{\t x'^{\a+1}}{z'},
 \end{equation}
by the quadratic behaviour of $v$ calculated on the small argument $\frac{\t}{x'}$.
Then we are left to prove that 
\begin{equation}
\label{giretti} 
\begin{aligned}
 \Big|\Big(\frac{\t}{z'}x'^{\a+1} \Big) \Big( \frac{\p_{x'}z'}{x'^{\a+1}}-\frac{\a+1}{x'^{\a+2}}z'\Big) \Big|\leq \s_0.
\end{aligned}
\end{equation}
The second term is easily estimated.  To conclude, we show that 
$\frac{\tau}{z'}|\frac{\p  z'}{\p x'}|\leq \s_0$. By a direct calculation  of
$\p_{x'} z'$, we are reduced to the proof that the inequality  
\begin{equation*}
\begin{aligned}
&\bigg| \tau  \Big((\a-1)h_2^2  (x'+h_1)^{\a-2}+\a h_1 (x'+h_1)^{\a-1}-((x'+h_1)^\a - x'^\a) \Big)\bigg|
\\&
\leq \s_0 \Big(h_2^2  (x'+h_1)^{\a-1} +h_1 (x'+h_1)^\a
 -\frac{ (x'+h_1)^{\a+1}-x'^{\a+1}}{\a+1}\Big)
\end{aligned}
\end{equation*}
holds for some $\s_0$  as small as we wish for small $\e_0$. The ratio
$\frac{\tau   (\a-1)h_2^2  (x'+h_1)^{\a-2}}{ h_2^2  (x'+h_1)^{\a-1}}$
 enjoys this property,  by the  estimate $\t\leq Cd\leq C\e_0
x'$. Thus, it suffices to prove the inequality with $h_2=0$. This can be
achieved by looking at the  following Taylor expansions in $h_1/x'$
\begin{equation*}
\begin{aligned}
&x'\Big( \a h_1 (x'+h_1)^{\a-1}-((x'+h_1)^\a - x'^\a)\Big) =\frac{\a(\a-1)}{2}x'^{\a-1}h_1^2 +
x'^{\a-1}h_1^2o(1)
\\&
h_1 (x'+h_1)^\a
 -\frac{ (x'+h_1)^{\a+1}-x'^{\a+1}}{\a+1}
 =\frac{\a }{2}x'^{\a-1}h_1^2 +
x'^{\a-1}h_1^2o(1),
\end{aligned}
\end{equation*}
where $o(1)\to 0$ as   $h_1/x'\to 0 $.
The proof of Lemma \ref{diciamo} is concluded. 
\end{proof}

\step{Estimate of $J_{5}$.} The (backward) curve connecting $u_4$ and $u_5$ is 
\[
 \gamma_5(t) = \Big( x'+\t, y'+\t+ t,
\frac{(x'+\t)^{\a+1}-x'^{\a+1}}{\a+1} +(x'+\t)^\a t \Big), \quad t\in [0,\tau],
\]
and we have 
\begin{equation*}
 J_5 =\int_B  \Big( 
 \int_0^\t  | X  f(\gamma_5(t))| dt 
 \Big)^p \frac{x^\a x'^\a dxdydx'dy'}{d^{p+2}x^{\a-1}}.
\end{equation*}
We change variable $t\mapsto s $ letting $(x'+\t)^\a s = (x'+\t)^\a t +
\frac{(x'+\t)^{\a+1}-x'^{\a+1}}{\a+1}
$.  Using $\t\leq Cd$, we get $0\leq s\leq C d$ and we have
\begin{equation*}
 J_5 \leq  \int_{B}
 \Big( 
 \int_0^{Cd}  \Big| X   f\Big( x'+\t, y'+\t+ \wh t,  (x'+\t)^{\a }s \Big)\Big|ds
 \Big)^p \frac{x^{\a+1} dxdydx' dy'}{d^{p+2}},
\end{equation*}
where $\wh  t =\wh t (s,x,x',y'-y ) =
s-\frac{1}{\a+1}\frac{(x'+\t)^{\a+1}-x'^{\a+1}}{ (x'+\t)^\a} $.

Next we pass from variables $x,y$  to variables $h_1 = x-x'$ and $h_2
=\sqrt{(y'-y)x}$. As in the previous cases, the Jacobian satisfies the estimate 
$ x  dx dy  \leq C |h| dh
$.
The unique solution $\wh \t=\wh \t(x',h)$ of~\eqref{cuccurullo}
does not depend on $y'$. Then, the function $\wh t=
s-\frac{1}{\a+1}\frac{(x'+\wh \t)^{\a+1}-x'^{\a+1}}{ (x'+\wh \t)^\a}
$  defined above, depends on $x', h_1, h_2$ but not on $y'$.  
Thus, after   the Minkowski inequality and   the change of variable $\ol y = 
y'+\wh\tau+\wh t$, we obtain
\begin{equation*}
\begin{aligned}
 J_5 & \leq C_\alpha 
\int _{\R^2} \Big( 
 \int_0^{C|h|}\Big [
\int_{x'>|h|/\e_0}
|X f(x'+\wh \t, \ol y , (x'+\wh \t)^\a s)|^p
x'^\a dx' d\ol y  
 \Big]^{1/p} ds 
 \Big)^p
 \frac{dh}{|h|^{p+1}}.
\end{aligned}
\end{equation*}

Finally, we exploit the transformation
$\ol x = x' +\wh\t(x',h)$. With a slight  modification of the argument used in
the estimate of $J_4$, see especially  \eqref{cuccurullo}, \eqref{girini} and
\eqref{giretti}, we see that $|\p_{x'}\wh \t(x', h)|<\frac 12$, if $\e_0$ is
small and $x'>|h|/\e_0$. Therefore we have a correct change of variable and moreover $x'^\a dx'
\simeq \ol x^\a d\ol x$. The argument is then concluded as in the  estimate
of~$J_4$.

\step{Estimate of  $J_6$.} We have to estimate the integral 
\begin{equation*}
\begin{aligned}
 J_6 = \int_B 
 \Big ( \int_0^{\t}\Big|Xf\Big (x'+t, y'+t,\frac{(x'
 +t)^{\a+1}- x'^{\a+1}}{\a+1}    \Big) \Big| dt \Big)^p\frac{x^\a x'^\a dxdy
dx'dy'}{d^{p+2}x^{\a-1}} .
 \end{aligned}
\end{equation*}
We use $\t\leq Cd$ and we change variables from $(x,y)$ to $h=(h_1,h_2)$ letting
$h_1=x-x'$ and $h_2=\sqrt{(y'-y)x'\;}$. 
Then $x dxdy \leq C|h| dh$, with $|h|\simeq d$ and we get  
\begin{equation*}
\begin{aligned}
 J_6 &\leq C_{\alpha}  
  \int_{\R^2} \int_{E_h}   
  \Big( \int_0^{C|h|}
  \Big|Xf\Big (x'+t, y'+t,\frac{ (x'
 +t)^{\a+1}- x'^{\a+1}}{\a+1}  \Big) \Big| dt
  \Big)^px'^\a dx'dy'  \frac{dh}{|h|^{p+1}},
\end{aligned}
\end{equation*}
where $E_h:=\{(x',y')\in\R^2: x'>|h|/\e_0\}$. Next we perform the change of variable  
\begin{equation*}
 t\mapsto s=\frac{1}{(x'+t)^\a}\frac{(x'+t)^{\a+1}-x'^{\a+1}}{\a+1}=:\phi_{x'}(t).
\end{equation*}
An explicit calculation gives 
$
\frac{d}{dt}\phi_{x'}(t)
\in \big[\frac{1}{ \a+1 }, 1 \big]$, for all  $t>0 $.
  Therefore $ \phi_{x'}(t)\simeq  t $, on $t\in[0, C|h|]$  and $ds\simeq dt$.
Denoting $\wh t=\phi_{x'}^{-1}(s)$, we get
\begin{equation*}
\begin{split}
 J_6 &\leq C_\alpha \int_{\R^2} 
   \int_{E_h}  
  \Big( \int_0^{C|h|} 
   |Xf  (x'+\wh t, y'+\wh t, (x'+\wh t )^\a s)  | ds
  \Big)^p  x'^\a dx'dy' \frac{dh}{|h|^{p+1}}.
\end{split}
\end{equation*}
An application of the Minkowski inequality  and  the change of variable 
$y'\mapsto \ol y =y'+\ol t $, where 
$\wh t$ does not depend on $y$, lead us to  
\begin{equation*}
J_6\leq C_{\alpha}  \int_{\R^2}  
    \frac{dh}{|h|^{p+1}}
   \bigg(\int_0^{C|h|} 
\bigg[
 \int_{  E_h}  
   |Xf  (x'+\wh t, \ol y, (x'+\wh t )^\a s)  |^p x'^\a dx' d\ol y
  \bigg]^{1/p} ds\bigg)^p.
\end{equation*}

Finally, we analyze the change of  variable $\ol x = x' + \phi_{x'}^{-1}(s)$. 
This is equivalent to 
$\phi_{x'}(\ol x- x')=s$ and using the  definition of $\phi$ we get
$
x'= (\ol x^{\a+1}- (\a+1)s \ol x^\a  )^{1/(\a+1)}  
$.
 Note 
 that $\ol x\geq x'$. Since $s\leq C d\leq C\e_0 x'\leq C\e_0 \ol x $, if $\e_0$ is small
 enough, we get 
 $x'^\a dx'\simeq x^\a dx$ and ultimately
 \begin{equation*}
  J_6 \leq C_{\alpha}  \int _{\R^2} \bigg(\int_0^{C|h|}
\bigg[\int _{\R^2} 
   |Xf  (\ol x , \ol y, |\ol x| ^\a s)  |^p |\ol  x| ^\a d\ol x  d\ol y
  \bigg]^{1/p}  ds\bigg)^p \frac{dh}{|h|^{p+1 }}.
 \end{equation*}
The estimate can be concluded as in the previous cases.

\begin{remark}\label{remarco} 
So far, we assumed that  $ y<y'$ and $x>x'>0$. If $y<y'$ and
$0<x < x'$ we add to the points $u=(x,y,0)$ and $v=u'=(x',y', 0)$ a third point
$u''=(2x-x',
 2y'-y, 0)$.
Then we have $d(u,u')\simeq d(u, u'')\simeq d(u',u'')$
and 
$
 \mu(B(u, d(u, u')))\simeq \mu(B(u, d(u, u'')))\simeq \mu(B(u', d(u', u''))).
$
Thus, the estimates in this case can be obtained as explained in Remark~\ref{remo}.

\end{remark}

\def\cprime{$'$} \def\cprime{$'$}
\providecommand{\bysame}{\leavevmode\hbox to3em{\hrulefill}\thinspace}
\providecommand{\MR}{\relax\ifhmode\unskip\space\fi MR }
\providecommand{\MRhref}[2]{%
  \href{http://www.ams.org/mathscinet-getitem?mr=#1}{#2}
}
\providecommand{\href}[2]{#2}


\end{document}